\documentclass{amsart}

 \newtheorem{theorem}{Theorem}[section]
 \newtheorem{lemma}[theorem]{Lemma}
\newtheorem{proposition}[theorem]{Proposition}
\newtheorem{corollary}[theorem]{Corollary}
\theoremstyle{definition}
 \newtheorem{definition}[theorem]{Definition}
 \theoremstyle{remark}
 \newtheorem{remark}[theorem]{Remark}
 \numberwithin{equation}{section}
 \newenvironment{proof-sketch}{\textit{Sketch of proof.~}}{}
  \renewenvironment{proof}{\textit{Proof.~}}{}
 \newcommand{\smartqed}{}


\begin{document}

\title{A variational study of some hadron bag models
}


\author{Lo\"{i}c Le Treust}



\maketitle

\begin{abstract}
We study, in this paper, some relativistic hadron bag models. We prove the existence of excited state solutions in the symmetric case and of a ground state solution in the non-symmetric case for the soliton bag and the bag approximation models by concentration compactness. We show that the energy functionals of the bag approximation model are $\Gamma$-limits of sequences of soliton bag energy functionals for the ground and excited state problems. The pre-compactness, up to translation, of the sequence of ground state solutions associated with the soliton bag energy functionals in the non-symmetric case is obtained combining the $\Gamma$-convergence theory and the concentration-compactness principle. Finally, we give a rigorous proof of the original derivation of the M.I.T. bag equations via a limit of bag approximation ground state solutions in the spherical case. The supersymmetry property of the Dirac operator is a key point in many of our arguments.
\end{abstract}

\section{Introduction and main results}
Quantum chromodynamics (QCD) is the theory of strong interaction and accounts for the internal structure of hadrons. At low-energy, the quarks are bound together to form baryons (protons, neutrons) and mesons. Nevertheless, the confinement mechanism has not yet been derived from the QCD equations. In order to study the hadronic properties, physicists introduced phenomenological models approximating the QCD equations in which the quarks are confined. Among them, the M.I.T. bag \cite{MIT061974,MIT101974,MIT101975,johnson} and the bag approximation models \cite{MIT061974,MIT101975} have been set in $1974$ and the soliton bag model \cite{Friedberg1977-1,Friedberg1977-2,Lee1992} in $1977.$  

The solutions of the equations of the soliton bag and the bag approximation models are critical points of non-linear functionals involving the Dirac operator. The mathematical techniques used to solved most equations of this type are different from the ones used in a non-relativistic framework (see the review paper of Esteban, Lewin and S\'{e}r\'{e} \cite{esteban2008}). Nevertheless, in our case, the supersymmetric properties of the Dirac operator with scalar potential \cite{Thaller1992} allow us to transform a strongly indefinite variational problem into a minimization one and then to use the direct method in the calculus of variation \cite{Struwe2008}. Since the functionals associated with the ground state problems of the soliton bag and the bag approximation models without symmetries are invariant under translations, we show the existence of solutions thanks to the concentration compactness method under some restrictions on the parameters of the models. The originality of the proofs relies on the fact that the usual concentration compactness inequalities are not satisfied anymore and we have to introduce different inequalities to overcome it. We also show the existence of ground and excited state solutions under some conditions on the parameters, when the wave functions are supposed to have some symmetries. These are the first rigorous proofs of existence for these two models. Actually, solving the bag approximation model is a shape optimization problem on finite perimeter sets of $\mathbb{R}^3$ which is related to the soliton bag model thanks to the gradient theory of phase transitions \cite{modicamortola1,modicamortola2,Modica1987,Sternberg,Braides1998,braides2006}. Indeed, we show that the energy functionals of the bag approximation model are $\Gamma$-limits of sequences of soliton bag model energy functionals for the ground and excited state problems. We combine the $\Gamma-$convergence theory and the concentration compactness method to get the pre-compactness, up to translation, of the sequence of ground state solutions associated with the soliton bag energy functionals in the non-symmetric case. As in the existence results, we have to introduce concentration compactness inequalities different from the classical ones. Bucur \cite{Bucur2000},  Bucur and Giacomini \cite{Bucur2010} have already studied shape optimization problems thanks to the concentration compactness method. But, to our knowledge, this is the first result using both the gradient theory of phase transition and the concentration compactness method. We also prove the pre-compactness of the sequences of excited state solutions associated with the soliton bag energy functionals in the symmetric case. Finally, we give a rigorous proof of the original derivation of the M.I.T. bag equations done by Chodos, Jaffe, Johnson, Thorn and Weisskopf \cite{MIT061974} via a limit of bag approximation ground state solutions in the spherical case.  These are the first proofs which rigorously establish the link between the soliton bag, the bag approximation and the M.I.T. bag models.

Let us now introduce the different models we study.
\subsection{Some bag models}

\subsubsection{The soliton bag model}
This model has been introduced by Friedberg and Lee \cite{Friedberg1977-1,Friedberg1977-2} and is sometimes called the Friedberg-Lee model.

Here, we look for a solution $(\psi_1,\dots,\psi_N,\phi)$ of the following system of equations:
\begin{equation}\label{equation}
\left\{
\begin{array}{lr}
H_{0}\psi_i+g\beta\phi\psi_i =\lambda_i\psi_i\;&\forall i=1,\dots,N,\\
\|\psi_i\|_{L^2}=1\;&\forall i=1,\dots,N,\\
-\Delta\phi+U'(\phi) + \sum_{\substack{i=1}}^Ng\psi_i^*\beta\psi_i=0,\\
\end{array}\right.
\end{equation}
where $N\in\mathbb{N}\backslash\{0\},~g>0,$ $\phi:\mathbb{R}^3\rightarrow\mathbb{R}$ and for all $i\in\{1,\dots,N\},$ $\psi_i:\mathbb{R}^3\rightarrow\mathbb{C}^4$.  $H_{0}= -i\alpha .\nabla+\beta m=-i\alpha_k\partial_k+\beta m$ is the Dirac operator in the Pauli-Dirac representation (see \cite{Thaller1992}) : $ \alpha = (\alpha_1,\alpha_2,\alpha_3),$
\begin{equation*}
\beta = \left(\begin{array}{cc}I_2&0\\0&-I_2\end{array}\right), \alpha_k = \left(\begin{array}{cc}0&\boldsymbol{\sigma}_k\\\boldsymbol{\sigma}_k&0\end{array}\right), \text{~for~} k=1,2,3,
\end{equation*}
with 
\begin{equation*}
\boldsymbol{\sigma}_1 = \left(\begin{array}{cc}0&1\\1&0\end{array}\right),\boldsymbol{\sigma}_2 = \left(\begin{array}{cc}0&-i\\i&0\\\end{array}\right),\boldsymbol{\sigma}_3 = \left(\begin{array}{cc}1&0\\0&-1\end{array}\right),
\end{equation*}
and $m>0$; $X^*$ denotes the complex conjugate of $X\in\mathbb{C}^4$. We have used here Einstein's convention for the summation.

Solutions of equations \eqref{equation} are called quasi-classical \cite{Friedberg1977-1,Friedberg1977-2}.

The potentials of the Dirac operator $H_0$ of the form $\beta\phi$ are called scalar potentials.  The fact that in our problems the potential of the Dirac operator is scalar is a key point in our study. Indeed, we will see below that the scalar potentials preserve the symmetry of the spectrum with respect to $0$ that the Dirac operator has. Let us remark moreover that $g\beta\phi$ acts like a mass term \cite{Lee1992} since
\[
H_0+g\beta\phi = -i\alpha.\nabla+\beta(m+g\phi).
\]   
This kind of potential is often considered in QCD \cite{Friedberg1977-1,Friedberg1977-2,Goldflam1982,MIT061974} to model the strong interactions between quarks. From the  physical point of view,  $\phi$ is a phenomenological scalar field that models the QCD vacuum and can be viewed as a representation of the quantum excitations of the self-interacting gluon field \cite{Goldflam1982}. 

$\psi_1,\dots,$ $\psi_N$ are the wave functions of the $N$ valence quarks. $N$ is fixed at $2$ for mesons and $3$ for baryons. $g$ is the positive coupling constant between the quark and the scalar fields. $\lambda_i$ is an eigenvalue of the Dirac operator with scalar potential $H_{0}+g\beta\phi$ and represents the energy of the $i^{th}$-relativistic particle in the scalar field $\phi$, so it has to be non negative for $\psi_i$ to be a physically admissible state (see chapter 1 of \cite{Thaller1992} for a physical interpretation of the negative part of the spectrum of a Dirac operator). Some of the $N$ particles can have the same wave function and this does not necessarily contradict Pauli's exclusion principle because quarks possess others quantum numbers such as color. Let us denote by $N_0$ the maximal number of particles possible with the same wave function $\psi$. We will always assume that the number of these particles is less than $N_0.$

Physicists \cite{Saly1983,Horn1986} have already studied numerically this problem for scalar potential with radial symmetry {\it i.e.} when $\phi$ is radial. In that case, the spin-orbit operator, the $z$-component of the angular momentum operator and the Dirac operator commute altogether. Hence, we will look for eigenfunctions of the Dirac operator with spherically symmetric potential that are also eigenfunctions of the spin-orbit operator and the $z$-component of the angular momentum operator. A particular ansatz is often chosen for the four-vector wave function \cite{Saly1983,Horn1986,mathieu1984,MIT061974}
\begin{equation}\label{soler}
\psi(x) = \left(\begin{array}{c}v(r)\left(\begin{array}{c}1\\0\end{array}\right)\\iu(r)\left(\begin{array}{c}\cos\theta\\\sin\theta e^{i\varphi}\end{array}\right)\end{array}\right)
\end{equation}
that is separable in the spherical coordinates $(r,\theta,\varphi)$ of $x$. This corresponds to an eigenfunction of the spin-orbit operator of eigenvalue $-1$ and of the $z$-component of the angular momentum operator of eigenvalue $1/2$ (see \cite[Section 4.6]{Thaller1992} for complete study of the Dirac operator with spherically symmetric potential). Actually, it is commonly admitted by physicists that the ground state of many problem involving the Dirac operator has to be searched among those functions, but to our knowledge, no rigorous proof ensures it. Once this choice is made, 
\[
	\psi^*\beta \psi = v^2-u^2
\]
becomes a radial function which in turn generates a radial potential $\phi$ in equations \eqref{equation}. The ansatz \eqref{soler} is well-known in physics and has been used in particular by Soler \cite{soler1970classical} to describe elementary fermions.

We denote by $H^{1/2}_{sym}(\mathbb{R}^3,\mathbb{C}^4)$ the set of the functions $\psi$ of this type which belong to $H^{1/2}(\mathbb{R}^3,\mathbb{C}^4)$ and $H^1_{rad}(\mathbb{R}^3,\mathbb{R})$ the radial functions of $H^1(\mathbb{R}^3,\mathbb{R})$. The problem of finding a solution of \eqref{equation} when we sought the scalar potentials among spherically symmetric functions and the quarks wave functions among functions of the form \eqref{soler} will be called the symmetric problem. Whereas, when no assumption is done on the form of the solution, we will say that this is the non-symmetric problem.

Saly, Horn, Goldflam and Wilets have already  found numerical ground \cite{Saly1983,Horn1986} and excited state \cite{Saly1984} solutions for the symmetric problem.

Throughout this paper, we will assume that $U:\mathbb{R}\rightarrow\mathbb{R}$ is a non-negative $\mathcal{C}^1$ function such that $U$ and its derivative $U'$ vanish at zero and
\[\tag{H1}\label{condition:U0} |U'(x)|\leq C(|x|+|x|^p) \text{~for $x\in\mathbb{R}$ with $1<p<5$},\]
\[\tag{H2}\label{condition:U2}  U(x)\geq cx^2 \text{~for all~} x\in\mathbb{R},\]
for some positive constants $c$ and $C$.

\begin{remark}
Physically, the presence of the constant $c$ in hypothesis \eqref{condition:U2} means that the scalar field $\phi$ has a mass whose value is at least $2c$.
\end{remark}

Our problem has indeed a variational structure: we look for a critical point of the $\mathcal{C}^1$ energy functional:
\begin{eqnarray*}
	\lefteqn{\mathcal{E}(\psi_1,\dots,\psi_N,\phi)}\\
	&& = \int_{\mathbb{R}^3} \left[\left(\sum_{\substack{i=1}}^{\substack{N}}(\psi_i,H_{0}\psi_i)+g\phi(\psi_i,\beta\psi_i)\right)+ \frac{|\nabla\phi|^2}{2}+U(\phi)\right]dx
\end{eqnarray*}
on the set $\{(\psi_1,\dots,\psi_N,\phi)\in H^{1/2}(\mathbb{R}^3,\mathbb{C}^4)^N\times H^1(\mathbb{R}^3,\mathbb{R}): \|\psi_i\|_{L^2}=1\}.$ 
$\lambda_i$ are the Lagrange multipliers associated with the $\|.\|_{L^2}$-constraints and $(~.~,~.~)$ is the complex scalar product.

\begin{remark}
Condition \eqref{condition:U0} is just a mathematical constraint for $\mathcal{E}$ to be well-defined and differentiable. This does not restrict the set of admissible potentials $U$ considered by the physicists \cite{Friedberg1977-1}.
\end{remark}

\subsubsection{The bag approximation}
This model has been introduced by Chodos, Jaffe, Johnson, Thorn and Weisskopf \cite{MIT061974,MIT101974} to derive the M.I.T. bag model as a limit case.

Here, the scalar field $\phi$ of the previous model is replaced by a characteristic function $\chi_\Omega$ but it still models the cavity where the quarks are encouraged to live. 

The Lagrangian of the bag approximation is:
\[
\mathcal{F}(\psi_1,\dots,\psi_N,\chi_\Omega)=\int_{\mathbb{R}^3} \left[\sum_{\substack{i=1}}^{\substack{N}}(\psi_i,H_{0}\psi_i)-g\chi_{\Omega}(\psi_i,\beta\psi_i)\right]dx+ aP(\Omega)+b|\Omega|,
\]
for $\psi_1,\dots,\psi_N$ in $H^{1/2}(\mathbb{R}^3,\mathbb{C}^4).$  $\psi_1,\dots,\psi_N$ still represent the quark wave functions. The characteristic function $\chi_\Omega$ of $\Omega$  belongs to $\{\chi_\omega \in BV(\mathbb{R}^3,\mathbb{R})\}$.  $|\Omega|$ denotes  the area of $\Omega$ and $P(\Omega)$ its perimeter. We will write, in this paper, the variation of a function  $\phi\in BV(\mathbb{R}^3,\mathbb{R})$ on a Borel set $\mathcal{A}$ by $|\nabla \phi|(\mathcal{A}),$ so that:
\[
	|\nabla\chi_\Omega|(\mathbb{R}^3)=P(\Omega).
\]
The constants $a,b,m,g$ are positive.

We look for  critical points of $\mathcal{F}$ on the set 
\[\{(\psi_1,\dots,\psi_N,\chi_\Omega)\in H^{1/2}(\mathbb{R}^3,\mathbb{C}^4)^N\times BV(\mathbb{R}^3,\mathbb{R}): \|\psi_i\|_{L^2}=1\}\]
that is to say, points that satisfy (see \cite{Giusti1984,Pierre2005}):

\begin{equation}\label{equation:cavity}\left\{
\begin{array}{llll}
\left(H_0-g\chi_\Omega\right)\psi_i&=\lambda_i\psi_i,\;&\forall i=1,\dots,N &\text{on}~\mathbb{R}^3\\
\|\psi_i\|_{L^2}&=1,\;&\forall i=1,\dots,N\\
a\mathcal{H}_\Omega+b-\sum_{i=1}^Ng\psi_i^*\beta\psi_i&=0\;&\mathcal{H}^{2}\text{-a.e. in}~\partial^*\Omega\\
\end{array}\right.\end{equation}
where $\mathcal{H}_\Omega$ is the mean curvature of $\partial^*\Omega,$  and $\mathcal{H}^{2}$ is the two-dimensional Hausdorff measure.

\subsubsection{The M.I.T. bag model}
The M.I.T. bag model is another model where the quark wave functions are perfectly confined in a bag \cite{MIT061974,MIT101974,MIT101975,johnson}. It has been widely studied and has lead to results fitting the experiments \cite{MIT101975}.

Let us, for the moment, introduce the equations in a fixed non-empty bounded regular open set $\Omega$ of $\mathbb{R}^3.$ In this paper, we will just consider the ground state problem, so that, we look for a single function $\psi$, solution of the following problem:

\begin{equation}\label{MITbagequation}\left\{
	\begin{array}{lll}	
		H_0\psi 			&= \lambda\psi		&\text{on}~\Omega\\
		-i\beta(\alpha.n)\psi 	&= \psi			&\text{on}~\partial\Omega\\
		\|\psi\|_{L^2(\Omega)}&=1,	
	\end{array}	
	\right.	
\end{equation}
where $\psi\in H^1(\Omega,\mathbb{C}^4),$  $\lambda>m$ and $n$ is the exterior normal to $\partial\Omega.$ 

When $\Omega=B(0,R)$, we look for an eigenfunction $\psi$ in $H^1_{sym}(\Omega,\mathbb{C}^4)$ i.e. of the form \eqref{soler}. In that case, the boundary condition becomes:

\[
	\begin{array}{ll}
		u=v	&\text{on}~\partial\Omega.
	\end{array}
\]

The problem of finding a good Lagrangian formulation for these equations has been widely studied \cite{johnson,Johnson1978}. This has been a motivation for the physicists to introduce other phenomenological models like the soliton bag model of Friedberg and Lee \cite{Friedberg1977-1,Friedberg1977-2} and the fractional bag model of Mathieu and Saly \cite{mathieu1984,mathieu1985}. Balabane, Cazenave and Vazquez \cite{balabane1990} already proved the existence of compactly supported ground state solutions for this latter model thanks to a shooting method.

\subsection{Variational formulations}
The main difficulty we have to face in the soliton bag and bag approximation models, is that the functionals considered are strongly indefinite: they are neither bounded from below nor from above and their critical points have an infinite Morse index. So, for now, we do not have any satisfactory formulations of the ground and excited state problems. 

The key point to overcome this in all the models, relies on a fine study of the Dirac operator with scalar potential.
\subsubsection{The soliton bag and bag approximation case}

\begin{lemma}\label{lemma:selfadjoint}
Let $\phi$ be in $L^p(\mathbb{R}^3,\mathbb{R})$, then $H_{\phi}=H_{0}+g\beta\phi$ is a self-adjoint operator on $L^2(\mathbb{R}^3,\mathbb{C}^4)$, with domain $H^{1}(\mathbb{R}^3,\mathbb{C}^4)$ and form-domain $H^{1/2}(\mathbb{R}^3,\mathbb{C}^4)$ whenever $3\leq p<+\infty$. It satisfies:
\[\sigma_{ess}(H_{\phi}) = \sigma_{ess}(H_{0}) = (-\infty, -m]\cup[m,+\infty).\]
\end{lemma}

The proof of this lemma is based on Kato-Rellich theorem and Weyl's criteria for essential spectrum and can be found in  \cite[Chapter 4]{Thaller1992}.

We denote by $E$ the Hilbert space where we search for the quark functions and $F$ the associated space for $\phi$ or $\chi_\Omega$ when no confusion is possible. $E\times F$ can be:
\[
	\begin{array}{l}
		H^{1/2}(\mathbb{R}^3,\mathbb{C}^4)\times H^1(\mathbb{R}^3,\mathbb{R}),\\
		H^{1/2}_{sym}(\mathbb{R}^3,\mathbb{C}^4)\times H^1_{rad}(\mathbb{R}^3,\mathbb{R}),\\
		H^{1/2}(\mathbb{R}^3,\mathbb{C}^4)\times \{\chi_\omega \in BV(\mathbb{R}^3,\mathbb{R})\},\\
		H^{1/2}_{sym}(\mathbb{R}^3,\mathbb{C}^4)\times \{\chi_\omega \in BV_{rad}(\mathbb{R}^3,\mathbb{R})\}.
	\end{array}
\] 
We define $E^+_{\phi} = \chi_{(0,+\infty)}(H_\phi)E$ where $\chi_{(0,+\infty)}$ is the characteristic function of $(0,+\infty)$, $d:=\dim( \ker(H_\phi))$ and for $k\in\mathbb{N}\backslash\{0\}:$
\[
	\lambda^k_+(H_\phi) := \left\{\begin{array}{ll}
		0&\text{if~} k\leq \frac{d}{2}\\
		\\
		\underset{\substack{V\subset E^+_{\phi}\\\dim V=k-\frac{d}{2}}}{\inf}~\underset{\substack{\|\psi\|_{L^2}=1\\\psi\in V}}{\sup}(\psi,H_\phi\psi)&\text{if}~k>\frac{d}{2}.
		\end{array}\right.
\]

\begin{remark}
The symmetry of the spectrum with respect to $0$ is actually true not only for the essential spectrum of Dirac operators with scalar potentials as in Lemma \ref{lemma:selfadjoint} but also for the whole spectrum. Moreover, we will see in Section \ref{secsupersym} that $d$ is pair so that the definition of $\lambda^k_+(H_\phi)$ make sense for all $k$. Then, we will get that
\[
	\{\pm\lambda^k_+(H_\phi) \}\cap(-m,m)
\]
are the eigenvalues of $H_\phi$ in $(-m,m)$ counted with multiplicity.

In this paper, the fact that the potentials are scalar is important. Indeed, the symmetry of the spectrum is not true anymore for the Dirac operator with an electric potential such as the Coulomb one \cite[Section 7.4]{Thaller1992}. This property is related to the supersymmetric operator theory which will be an essential tool in our study, for instance,  to give a simpler expression for $\lambda^k_+(H_\phi).$ 
\end{remark}

The well-defined minimization problems are then, for $1\leq k_1\leq\dots\leq k_N$ in the soliton bag model:

\begin{equation}\label{variationalformulation}
\inf\left\{\sum_{\substack{i=1}}^{\substack{N}}\lambda^{k_i}_+(H_\phi)+\int_{\mathbb{R}^3} \left[\frac{|\nabla\phi|^2}{2}+U(\phi)\right]dx
:\phi\in F\right\}
\end{equation} 
and in the bag approximation:
\begin{equation}\label{variationalformulationcavity}
\inf\left\{\sum_{\substack{i=1}}^{\substack{N}}\lambda^{k_i}_+(H_{-\chi_{\Omega}})+aP(\Omega)+b|\Omega|:\chi_{\Omega}\in F\right\}.
\end{equation} 
We get here a good formulation for the ground state problems when $k_1=\dots=k_N=1.$ The other cases are related to the exited states.

\subsubsection{The M.I.T. bag case}

Let $\Omega$ be the euclidean ball $B(0,R)$ of $\mathbb{R}^3$ with $R>0$.
Define
	\[
		\mathcal{D}(H_0) = \{\psi\in H^1_{sym}(\Omega,\mathbb{C}^4)~:~-i\beta(\alpha.n)\psi = \psi~\text{on}~\partial\Omega\}.
	\]
We call $(H_0,\mathcal{D}(H_0))$ the M.I.T. bag Dirac operator.
	
\begin{proposition}\label{prop:supersymmetryMIT1}
	The operator $(H_0,\mathcal{D}(H_0))$ is self-adjoint and there is a nondecreasing sequence of eigenvalues $(\lambda_n)_{n\geq1}\subset (m,+\infty)$ which tends to infinity such that:
	\[
		\sigma(H_0) = \{\dots,-\lambda_2,-\lambda_1\}\cup\{\lambda_1,\lambda_2,\dots\}.
	\]
	We denote for each $n,$ $ \lambda^n_{MIT}(\Omega) := \lambda_n.$
\end{proposition}
\begin{remark}
	The main ideas of the proof of Proposition \ref{prop:supersymmetryMIT1} will be given in the second section.
\end{remark}

The variational formulation for the ground state is:
\[
	\inf\left\{N\lambda^1_{MIT}(B(0,R))+aP(B(0,R))+b|B(0,R)|:R>0\right\},
\]
where $a,b>0$ and $N\leq N_0.$

\subsection{Existence results}
\subsubsection{The soliton case}
We get the following results:
\begin{theorem}\label{maintheorem}
Let $K\in\mathbb{N}\backslash\{0\}$ and $m>0$ be fixed. Assume that $U$ satisfies hypothesis \eqref{condition:U0} and \eqref{condition:U2}. There is $g_0>0$ such that if the coupling constant $g$ satisfies $g>g_0$ then, for any $1\leq k_1\leq\dots\leq k_N\leq K$, there exists a solution 
\[(\psi_1,\dots,\psi_N,\phi)\in H^{1/2}_{sym}(\mathbb{R}^3,\mathbb{C}^4)^N\times H^1_{rad}(\mathbb{R}^3,\mathbb{R})\]
 of equations \eqref{equation} with $\lambda_i = \lambda^{k_i}_+(H_\phi)\in (0,m)$ where $\phi$ is a minimum of problem \eqref{variationalformulation}. $g_0$ depends on $N,m,K$ and $U$.
\end{theorem}

Let us make some comments:
\begin{remark}
Friedberg and Lee \cite{Friedberg1977-1} derived some conditions on $m,~g$ and $U$ comparable to ours for the model to have soliton solutions. 
\end{remark}
\begin{remark}
They also assumed $U$ to be a non-negative polynomial of degree $4$ with two minima at $0$ and $-\phi_0<0$ such that $0=U(0)\leq U(-\phi_0).$ In many of their proofs, they considered condition \eqref{condition:U2} true. Nevertheless, most of the numerical works were performed by Saly, Horn, Goldflam and Wilets \cite{Saly1983,Saly1984,Horn1986} with $U(-\phi_0)=0$. Actually, the symmetry $U(0)=U(-\phi_0)$ seems not to prevent the scalar field to tend to $-\phi_0$ at infinity and this leads to some mathematical complications in the minimization. However, $c$ can be chosen as small as we want.
\end{remark}
This is the first rigorous proof of the existence of ground and excited states for wave functions of form \eqref{soler}. The symmetry of the functions leads to the compactness properties established by Strauss \cite{Strauss} and Lions \cite{Lions1982}. As we remark before, no result ensures that the ground state has to possess such a symmetry. So, in Theorem \ref{maintheorem2}, we prove the existence of a ground state with no assumption made on the form of the quark wave function with the help of the concentration compactness method.
\begin{theorem}\label{maintheorem2}
Let $m>0$ be fixed. Assume that $U$ satisfies hypothesis \eqref{condition:U0} and \eqref{condition:U2}. There is $g_0>0$ such that if the coupling constant $g$ satisfies $g>g_0$ then, there exist  
\[
(\psi,\phi)\in H^{1/2}(\mathbb{R}^3,\mathbb{C}^4)\times H^1(\mathbb{R}^3,\mathbb{R})
\]
such that $\phi$ is a minimum of problem \eqref{variationalformulation} for $k_1=\dots=k_N=1$ and
\begin{equation*}
\left\{
\begin{array}{lr}
H_{0}\psi+g\beta\phi\psi =\lambda\psi\\
\|\psi\|_{L^2}=1\\
\end{array}\right.
\end{equation*}
where $\lambda= \lambda^{1}_+(H_\phi)\in [0,m)$.  $g_0$ depends on $N,m$ and $U$.

If $0\notin\sigma(H_\phi)$, then $\phi$ satisfies
\[
-\Delta\phi+U'(\phi) + Ng\psi^*\beta\psi=0.
\]

\end{theorem}
\begin{remark}
The main problem that occurs when $0\in\sigma(H_\phi)$ is that $\phi\mapsto \lambda^1_+(H_\phi)$ is not necessarily G\^{a}teaux differentiable.
We will get in Corollary \ref{cor:solitongroundgeneral}  that $0\notin \sigma(H_\phi)$ under some restrictions on the parameters of the model. From the physical point of view, the most relevant parameters satisfy the requirements of this corollary \cite{Friedberg1977-1,Goldflam1982}.
\end{remark}

\subsubsection{The bag approximation}
The same method adapted to the $BV$ setting  gives us similar results for the bag approximation model. 
\begin{theorem}\label{maintheorem3}
Let $K\in\mathbb{N}\backslash\{0\}.$ Assume $g\in(0,m)$. There is a constant $\delta>0$ such that if:
\[\tag{H3}a,~b<\delta,\]
then, for any $1\leq k_1\leq\dots\leq k_N\leq K$, there exists a solution 
\[(\psi_1,\dots,\psi_N,\chi_\Omega)\in H^{1/2}_{sym}(\mathbb{R}^3,\mathbb{C}^4)^N\times BV_{rad}(\mathbb{R}^3,\mathbb{R})\]
 of equations \eqref{equation:cavity} with 
\[\lambda_i = \lambda^{k_i}_+(H_{-\chi_\Omega})\in (0,m)\]
where $\chi_\Omega$ is a minimum of problem \eqref{variationalformulationcavity}. 
\end{theorem}
\begin{theorem}\label{maintheorem4}
Assume $g\in(0,m)$. There is a constant $\delta>0$ such that if:
\[\tag{H3}a,~b<\delta,\]
then, there exists a solution $(\psi,\dots,\psi,\chi_\Omega)\in H^{1/2}(\mathbb{R}^3,\mathbb{C}^4)^N\times BV(\mathbb{R}^3,\mathbb{R})$ of equations \eqref{equation:cavity} with $\lambda =\lambda_i= \lambda^{1}_+(H_{-\chi_\Omega})\in (0,m)$ where $\chi_\Omega$ is a minimum of problem \eqref{variationalformulationcavity} for $k_1=\dots=k_N=1.$
\end{theorem}
\begin{remark}
In this case, the assumption $g\in(0,m)$ ensures that $0\notin \sigma(H_{-\chi_\Omega})$.
\end{remark}
\subsection{The bag approximation model as a $\Gamma$-limit of soliton bag models}

The following results show the link between the soliton bag and the bag approximation models and are based on the $\Gamma$-convergence theory. 

Let us consider first for $\epsilon>0$ and $b>0$ the following functionals:
\[
	{E}_\epsilon(\phi) = \left\{
		\begin{array}{ll}
			\int_{\mathbb{R}^3}\left(\epsilon|\nabla \phi|^2+W(\phi)/\epsilon+b|\phi|^2\right)dx &\text{if }~\phi\in H^1(\mathbb{R}^3,\mathbb{R})\\
			+\infty	&\text{otherwise}
		\end{array} 
	\right.
\]
and 
\[
	E_0(\phi) = \left\{
		\begin{array}{ll}
			aP(\Omega)+b|\Omega|	&\text{if}~\phi=-\chi_\Omega\in BV(\mathbb{R}^3,\mathbb{R})\\
			+\infty	&\text{otherwise,}
		\end{array}
	\right.
\]
where $W:\mathbb{R}\rightarrow \mathbb{R}^+$ is a $\mathcal{C}^1$ function  which satisfies $W^{-1}(\{0\})=\{-1,0\},$ \eqref{condition:U0} and $a=2\int_{-1}^0\sqrt{W(s)}ds$.

\begin{proposition}\label{gammalimit}
Assume that there are positive constants $c$ and $2<q$  such that:
\[
		\begin{array}{ll}
			W(t)\leq c(|t|^2+|t|^q) &\forall t.
		\end{array}
\]
Then, $E_\epsilon$ $\Gamma$-converges to $E_0$ in $L^2\cap L^{\frac{3(q+2)}{4}}.$

\end{proposition}
This proposition is an adaptation of the result of Modica and Mortola \cite{modicamortola1,modicamortola2} generalized by Modica \cite{Modica1987} (see also Sternberg \cite{Sternberg} or Braides \cite{Braides1998})  for the gradient theory of phase transitions in an unbounded setting. Its proof strongly uses the one of \cite{Sternberg}.

Let us introduce for $\epsilon>0$: 
\[
	\mathcal{E}_\epsilon(\phi)=\left\{
	\begin{array}{ll}
		N\lambda^1_+(H_\phi)+E_\epsilon(\phi) 	&\text{if }\phi \in H^1(\mathbb{R}^3,\mathbb{R})\\
		+\infty																					&\text{otherwise}
	\end{array}
	\right.
\]
and 
\[
	\mathcal{E}_0(\phi)=\left\{
	\begin{array}{ll}
		N\lambda^1_+(H_{-\chi_\Omega})+E_0(-\chi_\Omega) 	&\text{if }\phi=-\chi_\Omega \in BV(\mathbb{R}^3,\mathbb{R})\\
		+\infty																					&\text{otherwise.}
	\end{array}
	\right.
\]

\begin{theorem}\label{gammalimit2}
	Assume that the condition of Proposition \ref{gammalimit} is true and that $g\in(0,m)$. Then, $\mathcal{E}_{\epsilon}$ $\Gamma-$ converges to $\mathcal{E}_0$ in $L^2\cap L^{\frac{3(q+2)}{4}}.$

	Let us assume besides that there are $c>0,$ $t_1<-1<t_2<0$ such that $W$ satisfies:
	\[
			W(t)\geq c|t|^q
	\]
	for all $t\notin(t_1,t_2)$ and 
	\begin{equation}\label{lc}
		l_c=\inf\{\mathcal{E}_0(\phi):~\phi=-\chi_\Omega \in BV\}<Nm.
	\end{equation}
	Then, there is $\epsilon_0>0$ such that for all $0<\epsilon<\epsilon_0,$ the problem 
	\begin{equation}\label{lsepsilon}
	l_s^\epsilon=\inf\{\mathcal{E}_\epsilon(\phi):~\phi \in H^1\}<Nm
	\end{equation}
	has a minimum $\phi_\epsilon.$ There is a subsequence such that, up to translation,  we have:
	\[
		\left\{\begin{array}{l}
			\mathcal{W}\circ\phi_{\epsilon_n}\rightarrow \mathcal{W}\circ(-\chi_\Omega)~\text{strictly in}~BV\\
			\phi_{\epsilon_n}\rightarrow (-\chi_\Omega)~\text{strongly in}~L^p~\text{for}~p\in[2,\frac{3(q+2)}{4}]\\
			l_s^{\epsilon_n} \rightarrow l_c
		\end{array}\right.
	\]
	where $-\chi_\Omega$ is a minimum  of the problem \eqref{lc} and $\mathcal{W}:t\mapsto  2\int_0^t\sqrt{W(s)}ds.$
\end{theorem}

The constant $\epsilon$ in the functionals can be obtained by scale change in some soliton bag functional.

\begin{remark}
	The physicists \cite{Friedberg1977-1} actually considered potentials $U$ in the soliton bag model of the form
	\[
		U:\phi\mapsto W(\phi)+b|\phi|^2
	\] 
satisfying the conditions of Theorem \ref{gammalimit2}.
\end{remark}
\begin{remark}
	Goldflam and Wilets \cite{Goldflam1982} studying the dependence of the numerical solutions on the parameters exhibit behaviors of the $\phi$ field similar to the ones of the Modica-Mortolla problem \cite{modicamortola1,modicamortola2,Modica1987,Sternberg,Braides1998}. Nevertheless, this is the first result which shows clearly the link between the two models we studied.
\end{remark}
\begin{remark}
	The main difficulty here is that the problems are set in an unbounded domain. We overcome this combining the $\Gamma$-convergence theory and the concentration compactness method.
\end{remark}
In the next corollary, we give conditions on the parameters of the soliton bag model that ensure that $0$ does not belong to $\sigma(H_\phi)$ where $\phi$ is a minimum of the ground state problem. Hence, $\phi$ satisfies the last equation of system \eqref{equation}.

\begin{corollary}\label{cor:solitongroundgeneral}
Assume the hypothesis of Theorem \ref{gammalimit2}  true. Then, there are $\epsilon_0>0$ and for all $\epsilon\in (0,\epsilon_0)$  a minimizer $\phi$ of problem \eqref{lsepsilon}, a function $\psi\in H^1(\mathbb{R}^3,\mathbb{C}^4)$ which satisfy
\begin{equation*}
\left\{
\begin{array}{lr}
H_{0}\psi+g\beta\phi\psi =\lambda\psi&a.e.~\text{in}~\mathbb{R}^3\\
\|\psi\|_{L^2}=1\\
-\epsilon\Delta\phi+\frac{W'(\phi)}{\epsilon}+2b\phi + Ng\psi^*\beta\psi=0,&a.e.~\text{in}~\mathbb{R}^3\\
\end{array}\right.
\end{equation*}
where $\lambda=\lambda^1_+(H_{\phi})>0$.
\end{corollary}
\begin{remark}
From the physical point of view, the most relevant parameters for the soliton bag model satisfy these requirements \cite{Friedberg1977-1}. Indeed,  Friedberg and Lee considered a potential $U$ composed of a two well potential $W$ and a mass term. The two well potential and the restriction $\epsilon\in(0,\epsilon_0)$ are introduced so as to force the scalar field $\phi$ to be almost a characteristic function\cite{Friedberg1977-1,Goldflam1982}.  
\end{remark}
We get a result to the one of Theorem \ref{gammalimit2} in the symmetric case. Let $1\leq k_1\leq\dots\leq k_N\leq K$ be integers. We define for $\epsilon>0$: 
\[
	\mathcal{E}_{\epsilon,k_1,\dots,k_N}(\phi)=\left\{
	\begin{array}{ll}
		\sum_{i=1}^N\lambda^{k_i}_+(H_\phi)+E_\epsilon(\phi) 	&\text{if }\phi \in H^1_{rad}(\mathbb{R}^3,\mathbb{R})\\
		+\infty																					&\text{otherwise}
	\end{array}
	\right.
\]
and 
\[
	\mathcal{E}_{0,k_1,\dots,k_N}(\phi)=\left\{
	\begin{array}{ll}
		\sum_{i=1}^N\lambda^{k_i}_+(H_{-\chi_\Omega})+E_0(-\chi_\Omega) 	&\text{if }\phi=-\chi_\Omega \in BV_{rad}(\mathbb{R}^3,\mathbb{R})\\
		+\infty																					&\text{otherwise.}
	\end{array}
	\right.
\]

\begin{theorem}\label{gammalimit3}
	 Assume the condition of Proposition \ref{gammalimit} true and that $g\in(0,m)$. Then, $\mathcal{E}_{\epsilon,k_1,\dots,k_N}$ $\Gamma-$ converges to $\mathcal{E}_{0,k_1,\dots,k_N}$ in $L^2\cap L^{\frac{3(q+2)}{4}}.$

	Let us assume besides that there are $c>0,$ $t_1<-1<t_2<0$ such that $W$ satisfies:
	\[
			W(t)\geq c|t|^q
	\]
	for all $t\notin(t_1,t_2)$ and 
	\[
		\inf\{\mathcal{E}_{0,K,\dots,K}(\phi):~\phi=-\chi_\Omega \in BV_{rad}\}<Nm.
	\]
	Then, there is $\epsilon_0>0$ such that for all $0<\epsilon<\epsilon_0,$ the problem 
	\begin{equation}\label{lsepsilonex}
	l_s^\epsilon(k_1,\dots,k_N)=\inf\{\mathcal{E}_{\epsilon,k_1,\dots,k_N}(\phi):~\phi \in H^1_{rad}\}<Nm
	\end{equation}
	has a minimum $\phi_\epsilon.$ There is a subsequence such that:
	\[
		\left\{\begin{array}{l}
			\mathcal{W}\circ\phi_{\epsilon_n}\rightarrow \mathcal{W}\circ(-\chi_\Omega)~\text{strictly in}~BV\\
			\phi_{\epsilon_n}\rightarrow (-\chi_\Omega)~\text{strongly in}~L^p~\text{for}~p\in[2,\frac{3(q+2)}{4}]\\
			l_s^{\epsilon_n}(k_1,\dots,k_N)\rightarrow l_c(k_1,\dots,k_N)
		\end{array}\right.
	\]
	where $-\chi_\Omega$ is a minimum  of problem:
	\begin{equation}\label{lcex}
		l_c(k_1,\dots,k_N) = \inf\{\mathcal{E}_{0,k_1,\dots,k_N}(\phi):~\phi=-\chi_\Omega \in BV_{rad}\}.
	\end{equation}
\end{theorem}

\subsection{The M.I.T. bag limit}

We study in this paper the M.I.T. bag ground state problem in the spherical case, i.e. when the open set $\Omega$ is a ball and the wave function belongs to $H^1_{sym}(\Omega,\mathbb{C}^4)$. Indeed, our main goal in this section is to give a rigorous proof of the original derivation of the M.I.T. bag equations done by Chodos, Jaffe, Johnson, Thorn and Weisskopf \cite{MIT061974} via a limit of bag approximation ground state solutions in the spherical case.

We assume in this section that $a,b\in\mathbb{R}^+$ and $\max\{a,b\}>0$.

\begin{proposition}\label{prop:MITgroundstate}
	There is a minimizer $R>0$ of 	
	\[
	\inf\left\{N\lambda^1_{MIT}(B(0,R))+aP(B(0,R))+b|B(0,R)|:R>0\right\}.
	\]
\end{proposition}

\begin{theorem}\label{theo:MITlimit}
	 Let $(M_n)_n\subset (0,+\infty)$ be an increasing sequence such that:
	\[
		\underset{n\rightarrow+\infty}{\lim} M_n=+\infty.
	\]
	There are $C_0,~ n_0>0,$ and for $n\geq n_0,$  a minimizer $R_n>0$ of 
	\[
		l_n:=\inf\left\{N\lambda^{1}_+(H^n_{B(0,R)})+aP(B(0,R))+b|B(0,R)|:R>0\right\}\leq C_0,
	\]
	a function $\psi_n\in H^{1}_{sym}(\mathbb{R}^3,\mathbb{C}^4)$ satisfying:
	\[
		\left\{\begin{array}{ll}
			H^n_{B(0,R_n)}\psi_n 	&= \lambda^{1}_+(H^n_{B(0,R_n)})\psi_n\\
			\|\psi_n\|_{L^2}			&=1,
		\end{array}\right.
	\]
	where $H^n_{\Omega}=-i\alpha.\nabla+\beta(m\chi_\Omega+M_n\chi_{\Omega^c}),$ such that, up to a subsequence:
	\[
		\left\{\begin{array}{l}
			R_{n}\rightarrow R>0,\\
			l_n\rightarrow \inf\left\{N\lambda^1_{MIT}(B(0,r))+aP(B(0,r))+b|B(0,r)|:r>0\right\},\\
			\psi_{n}\rightarrow \psi\chi_{B(0,R)}~\text{in}~L^2(\mathbb{R}^3)~ \text{and in}~L^\infty(B(0,R+\epsilon)^c\cup B(0,R-\epsilon))				\end{array}\right.
	\]
	for all $0<\epsilon<R/2$. $R$ comes from Proposition \ref{prop:MITgroundstate}, the function $\psi\in H^1_{sym}(B(0,R),\mathbb{C}^4)$ satisfies:
	\[
	\left\{
	\begin{array}{lll}	
		H_0\psi 			&= \lambda^1_{MIT}(B(0,R))\psi		&\text{on}~B(0,R)\\
		-i\beta(\alpha.n)\psi 	&= \psi			&\text{on}~\partial B(0,R)\\
		\|\psi\|_{L^2(B(0,R))}&=1.
	\end{array}	
	\right.	
	\]
\end{theorem}

\begin{remark}
Chodos, Jaffe, Johnson, Thorn and Weisskopf impose to the ground state cavity to be a ball, just as in Theorem \ref{theo:MITlimit}. Nevertheless, if we want to remove this restriction, some difficulties occur. We will point out in our proof where the problems arise. 
\end{remark}

The key point of all this paper is the use of supersymmetry properties of the Dirac operator studied in the second section. We give in the third section, some auxiliary results related to the continuity of the eigenvalues of $H_\phi$ in $\phi$. We prove the existence theorems for the soliton bag and the bag approximation models in the symmetric case in the fourth section. In the fifth section, the pre-compactness, up to translation, of minimizing sequences for the existence Theorems \ref{maintheorem2} and \ref{maintheorem4} follows from the concentration-compactness method. Supersymmetry allows us to get rid of the problems occurring with the constraints on the sign of the eigenvalues $\lambda$ of the operators $H_\phi$ and gives the binding inequalities necessary in the concentration-compactness argument. The sixth section is related to the proofs of Proposition \ref{gammalimit}, Theorems \ref{gammalimit2} and \ref{gammalimit3}, which are based on $\Gamma$-convergence and concentration compactness method. Finally, we give the first rigorous proof of the derivation of the M.I.T. bag equations in the last section.

\section{Supersymmetry of the Dirac operator and spectral properties} \label{secsupersym}
The variational formulations \eqref{variationalformulation} and \eqref{variationalformulationcavity} are not satisfactory because the definitions of the eigenvalues $\lambda^k_+(H_\phi)$ and $\lambda^k_{MIT}(\Omega)$ are not easy to handle for $k>0,$ $\Omega$ an open set of $\mathbb{R}^3$ and $\phi\in L^p$ for some $p\geq 3.$ Nevertheless, the supersymmetry theory for Dirac operators with scalar potentials will allow us to overcome these problems. We strongly use in this part the introduction to the theory of Thaller \cite[Chapter 5]{Thaller1992}. 

\begin{definition}
Let $\tau$ be a non-trivial unitary involution on a Hilbert space $\mathcal{H}$. A self-adjoint operator $Q$ on $\mathcal{H}$ with domain $\mathcal{D}(Q)$ is a supercharge with respect to $\tau$ if $\tau\mathcal{D}(Q)\subset\mathcal{D}(Q)$ and $\tau Q = -Q\tau$ on $\mathcal{D}(Q)$.
\end{definition}

\subsection{The Dirac operator with scalar potential on $\mathbb{R}^3$}
We begin by a study of Dirac operators on $L^2(\mathbb{R}^3,\mathbb{C}^4)$ with a special type of potentials, the scalar ones. Let $\phi$ be in $L^p(\mathbb{R}^3,\mathbb{R})$ for some $3\leq p<+\infty$. We define: 
\[
T:=\frac{1}{\sqrt{2}}\left(\begin{array}{cc} {I}_2&i{I}_2\\i{I}_2&{I}_2\end{array}\right), ~ D_{\phi} := -i\boldsymbol{\sigma}.\nabla+i(m+g\phi), ~D^*_{\phi} := -i\boldsymbol{\sigma}.\nabla-i(m+g\phi)
\]
and
\[
Q_\phi:= TH_{\phi}T^{-1}= \left(\begin{array}{cc}0&D_{\phi}^*\\D_{\phi}&0\end{array}\right)
\]
where we recall that $H_\phi = H_0+g\beta\phi$.
Then, by Lemma \ref{lemma:selfadjoint}, $Q_\phi$ is a supercharge with respect to the involution $\beta$ whose domain is $\mathcal{D}(Q_\phi) = T\mathcal{D} (H_\phi)$. It can be either $H^1(\mathbb{R}^3,\mathbb{C}^4)$ in the non-symmetric case or $H^1_{sym}(\mathbb{R}^3,\mathbb{C}^2)^2$ in the symmetric case where we denote by $H^1_{sym}(\mathbb{R}^3,\mathbb{C}^2)$ the subset of $H^1(\mathbb{R}^3,\mathbb{C}^2)$ whose functions are of the form:
\[
x\mapsto\left(\begin{array}{c}
-v(r)+u(r)cos\theta\\
u(r)sin\theta e^{i\varphi}
\end{array}\right),
\]
in the spherical coordinates $(r,\theta,\varphi)$ of $x$.

\begin{remark}
The fact that the potential $g\beta\phi$ of $H_\phi$ is scalar is a key point for 
$Q_\phi$ to be a supercharge with respect to $\beta$.
\end{remark}

$D_\phi$ is a closed operator on $L^2(\mathbb{R}^3,\mathbb{C}^2)$ with domain $H^1$ such that $D_{\phi}^*$ is its adjoint and vice versa. Let us remark moreover that:
\[
Q^2_\phi = \left(\begin{array}{cc}{D_{\phi}}^*D_\phi&0\\0&D_{\phi}{D_{\phi}}^*\end{array}\right),
\]
is a self-adjoint operator on $L^2$ with domain
\[
\mathcal{D}(Q^2_\phi)=\{\psi\in H^1: Q_\phi\psi\in H^1\},
\]
which can be different from $H^2$ if $\phi$ is not regular enough.

In the following lemma, we show that under some conditions on $\phi$, $0$ is not in the spectrum of $H_\phi.$
\begin{lemma}\label{supersymmetry_kernel}
Assume that $\phi\in L^p(\mathbb{R}^3,\mathbb{R})$ with $3\leq p<+\infty$ satisfies $m+g\phi\geq0$, then 
\[
	\ker(Q_\phi) = \ker({Q_\phi}^2) = \ker({D_{\phi}}^*D_\phi)\oplus \ker(D_{\phi}{D_{\phi}}^*)=\{0\}.
\]
\end{lemma}

\begin{proof}\smartqed
Let us assume that there exists $\omega\in H^1(\mathbb{R}^3,\mathbb{C}^2)$ such that 
\[
-i\sigma.\nabla\omega \pm i(m+g\phi)\omega.
\]
We get that
\[
\int_{\mathbb{R}^3}(m+g\phi)|\omega|^2dx = \pm\int_{\mathbb{R}^3}(\omega^*\sigma.\nabla\omega)dx = \pm\int_{\mathbb{R}^3}div(\omega^*\sigma\omega)dx = 0,
\]
hence, we have $(m+g\phi)|\omega|^2=0$ almost everywhere. Moreover, we have
\[
0=\int_{\mathbb{R}^3}(m+g\phi)^2|\omega|^2dx = \int_{\mathbb{R}^3}|\sigma.\nabla\omega|^2dx =\int_{\mathbb{R}^3}|\nabla\omega|^2dx.
\]
Thus, we have the result.
\qed\end{proof}
\begin{remark}
Let us remark that in the bag approximation case, if we assume that $g\in(0,m)$ we get for any $\chi_\Omega\in BV(\mathbb{R}^3,\mathbb{R})$ that $m+g(-\chi_\Omega)\geq0$. So, we have that $\ker(H_\phi)=\{0\}$. 
\end{remark}
The same result is true for the soliton bag model in the symmetric case. 
\begin{lemma}\label{simplicity}
Assume that $\phi\in L^{3+\epsilon}_{rad}$ for $\epsilon>0$. Then, every eigenvalue of $H_\phi$ is simple and $0\notin \sigma(H_\phi)$.
\end{lemma}
\begin{proof}\smartqed
Let $\lambda\in\sigma_p(H_\phi),$ by a standard bootstrap argument, every associated eigenvector belongs to $W^{1,q}$ for any  $q\geq2$ and so to $L^\infty.$ As we work with functions of the form:
\begin{equation*}
\psi(x) = \left(\begin{array}{c}v(r)\left(\begin{array}{c}1\\0\end{array}\right)\\iu(r)\left(\begin{array}{c}\cos\theta\\\sin\theta e^{i\varphi}\end{array}\right)\end{array}\right),
\end{equation*}
$(u,v)$ is a solution of the following system of equation:
\[
\left\{\begin{array}{lcr}
v'(r)&=&-(\lambda+m+g\phi(r))u(r),\\
u'(r)+\frac{2u(r)}{r}&=& (\lambda-m-g\phi(r))v(r)
\end{array}\right.
\]
and satisfies:
\[\left\{\begin{array}{l}
	u(r) = \frac{1}{r^2}(\int_0^rs^2v(s)(\lambda-m-g\phi(s))ds),\\
	v(r)=v(0)-\int_0^ru(s)(\lambda+m+g\phi(s))ds.
\end{array}\right.\]
By a contraction mapping argument \cite{balabane1990}, the solution is uniquely determined by $v(0)$. So the set of the eigenvectors of $H_\phi$ of eigenvalue $\lambda$ is of dimension $1$.

It remains to prove that $0$ does not belong to $\sigma(H_\phi)$. Let us assume by contradiction that there is $\psi\in H^1_{sym}(\mathbb{R}^3,\mathbb{C}^4)\backslash\{0\}$ such that $H_\phi\psi=0$. Then, we get that $Q_\phi T\psi=0$ and 
\[
T\psi = \left(\begin{array}{c}\omega_1\\\omega_2\end{array}\right)
\]
where $\omega_1,\omega_2\in H^1_{sym}(\mathbb{R}^3,\mathbb{C}^2)\backslash\{0\}$. Thus, we have
\[
	\left(\begin{array}{c}\omega_1\\0\end{array}\right),\left(\begin{array}{c}0\\\omega_2\end{array}\right)\in\ker Q_\phi
\]
and $0$ is not simple. This is impossible so $0\notin\sigma(H_\phi)$.
\qed\end{proof}

Let us give another lemma  in the non-symmetric case which study the case where $0\in\sigma(H_\phi)$ and ensures that $\lambda^k_+(H_\phi)$ is well-defined for all $k$.
\begin{lemma}\label{0eigen}
Assume that $\phi\in L^p(\mathbb{R}^3,\mathbb{R})$ for $3\leq p <+\infty$. Then, we have
\[
	\dim(\ker({D_\phi}^*{D_\phi})) = \dim(\ker({D_\phi}{D_\phi}^*))=\dim(\ker(H_\phi))/2=:d/2.
\]
\end{lemma}
\begin{proof}\smartqed
	We have
	\[
		\ker(Q_\phi) = \ker({D_{\phi}}^*D_\phi)\oplus \ker(D_{\phi}{D_{\phi}}^*).
	\]
	We suppose that $0\in\sigma({D_\phi}^*{D_\phi})$ then, there is 
	\[
	\omega=\left(\begin{array}{c}u\\v\end{array}\right)\in H^1(\mathbb{R}^3,\mathbb{C}^2)\backslash\{0\}
	\]
	 such that
	\[
	-\sigma.\nabla\omega+(m+g\phi)\omega=0.
	\]	
	We get
	\begin{eqnarray*}
		-\partial_1v+i\partial_2v-\partial_3u+(m+g\phi)u=0\\
		-\partial_1u-i\partial_2u+\partial_3v+(m+g\phi)v=0
	\end{eqnarray*}
	and
	\begin{eqnarray*}
		-\partial_1\tilde{v}+i\partial_2\tilde{v}-\partial_3\tilde{u}-(m+g\phi)\tilde{u}=0\\
		-\partial_1\tilde{u}-i\partial_2\tilde{u}+\partial_3\tilde{v}-(m+g\phi)\tilde{v}=0
	\end{eqnarray*}
	where
		\[
	\tilde{\omega}=\left(\begin{array}{c}\tilde{u}\\\tilde{v}\end{array}\right) := \left(\begin{array}{c}\overline{v}\\-\overline{u}\end{array}\right).
	\]
	Hence, we get
		\[
	-\sigma.\nabla\tilde\omega-(m+g\phi)\tilde\omega=0
	\]	
	so $0\in\sigma({D_\phi}{D_\phi}^*)$.  This ensures that 
	\[
		\dim(\ker({D_\phi}^*{D_\phi})) \leq \dim(\ker({D_\phi}{D_\phi}^*))
	\]
	A similar argument gives us the inverse inequation and the result follows.
\qed\end{proof}
We are now able to write down the Foldy-Wouthuysen representation of our supercharge operator. This allows us to give simpler expressions for the eigenvalues.
\begin{theorem}\label{supersymmetry_foldy}
Define the unitary transformations : 
\[S = (\sqrt{D_{\phi}{D_{\phi}}^*})^{-1}D_{\phi} = D_\phi\sqrt{{D_{\phi}}^*D_\phi}^{-1}\text{~of~}\ker(D_\phi)^\perp~\text{onto}~\ker({D_\phi}^*)^\perp\]
and  
\[
sgn~ Q_\phi = \left(\begin{array}{cc}0&S^*\\S&0\end{array}\right) \text{~of~}\ker(Q_\phi)^\perp.
\]
We denote the Foldy-Wouthuysen transformation
\[
U_{FW} := \left\{\begin{array}{ll}
	\frac{1}{\sqrt{2}}(1+\beta(sgn~ Q_\phi)) &\text{~on~}\ker(Q_\phi)^\perp\\
	1&\text{~on~}\ker(Q_\phi).
	\end{array}\right.
\]
Then, we have:
\[
U_{FW}Q_\phi U_{FW}^*=\beta|Q_{\phi}| = \left(\begin{array}{cc}\sqrt{{D_{\phi}}^*D_\phi}&0\\0&-\sqrt{D_{\phi}{D_{\phi}}^*}\end{array}\right),
\]
and ${D_{\phi}}^*D_\phi=SD_{\phi}{D_{\phi}}^*S^*$ on $\ker(D_\phi)^\perp$. Moreover, we have:
\[
m^2 = \inf\sigma_{ess}({D_{\phi}}^*D_{\phi})
\]
and
\[
\sigma(H_\phi)= \left((-\infty,-m]\cup\underset{k\geq1}{\bigcup}\left\{\pm\sqrt{\lambda^{k}({D_{\phi}}^*D_{\phi})}\right\}\cup[m+\infty)\right)
\]
where 
\[
\lambda^k_+(H_\phi)= \sqrt{\lambda^{k}({D_{\phi}}^*D_{\phi})}:=\inf_{\substack{V\subset H^1,\\\dim V=k}}\; \sup_{\substack{\omega\in V,\\\|\omega\|_{L^2}=1}}\|D_{\phi}\omega\|_{L^2}.
 \]
\end{theorem}

The proof can be found in Thaller \cite[Theorem 5.5, Corollary 5.6]{Thaller1992} and we give here a sketch of proof for the reader's convenience.

\begin{proof-sketch}\smartqed
On $\ker(Q_\phi)^\perp$, we easily get that $U_{FW}^{-1} = U_{FW}^* = \frac{1}{\sqrt{2}}(1-\beta(sgn~ Q_\phi))$ and 
\[
U_{FW}Q_\phi U_{FW}^* = \beta(sgn~ Q_\phi) Q_\phi = \beta|Q_\phi|,
\] 
since $Q_\phi$ and $sgn~ Q_\phi$ commute. $sgn~ Q_\phi$ commutes with $Q_\phi^2$ too so, we get that ${D_{\phi}}^*D_{\phi} = S{D_{\phi}}^*D_{\phi}S.$ 
\qed\end{proof-sketch}

\begin{remark}
All the results of Theorem \ref{supersymmetry_foldy} are also true in the symmetric case if we replace the spaces 
\[L^2(\mathbb{R}^3,\mathbb{C}^2), ~H^1(\mathbb{R}^3,\mathbb{C}^2),\dots\]
involved with those denoted by 
\[L^2_{sym}(\mathbb{R}^3,\mathbb{C}^2),H^1_{sym}(\mathbb{R}^3,\mathbb{C}^2),\dots\]
composed of the functions of the form:
\[
x\mapsto\left(\begin{array}{c}
v(r)+u(r)cos\theta\\
u(r)sin\theta e^{i\phi}
\end{array}\right),
\]
where $(r,\theta,\phi)$ are the spherical coordinates of $x$.
\end{remark}

We give now conditions on the parameters of the soliton bag and the bag approximation models that ensures that the operator $H_\phi$ associated with any minimizer $\phi$ of \eqref{variationalformulation} or \eqref{variationalformulationcavity} has enough eigenvalues in $[0,m)$ counted with multiplicity. The following lemmas are true both in the symmetric and in the general case. 

\begin{lemma}\label{supersymmetry_estimationsoliton}
Let $k\in\mathbb{N}\backslash\{0\}$ and $m>0$ be fixed. There exists $g_0>0$ such that for $g>g_0$ we have
\begin{equation*}
0<ls_k := {\inf}\left\{N\lambda^{k}_+(H_{\phi})+\int_{\mathbb{R}^3}[\frac{|\nabla\phi|^2}{2}+U(\phi)]dx: \phi\in F\right\}<Nm.
\end{equation*}
$g_0$ depends on $N,k,m$ and $U$.
\end{lemma}

\begin{proof}\smartqed

For $0<R<R'$, $0<\epsilon$, let $\phi_{R,R'}\in\mathcal{C}^{\infty}_0(\mathbb{R}^3,[-m/g,0])$ be a radial function such that 
\[\|\nabla\phi_{R,R'}\|_{L^{\infty}}\leq\frac{m+\epsilon}{g(R'-R)}\]
 and
\begin{equation*}
\phi_{R,R'}(x) = \left\{\begin{array}{ccl}-m/g& for &x\in B(0,R)\\ 0& for &x\in \mathbb{R}^3\backslash B(0,R').\end{array}\right.
\end{equation*}
Let $\omega_R\in H^1_{sym}(\mathbb{R}^3,\mathbb{C}^2)$ be such that $supp(\omega_R)\subset B(0,R)$ and $\|\omega_R\|_{L^2}=1,$ where $B(0,R)$ is the ball centered at $0$ in $\mathbb{R}^3$ of radius $R$.
Then, we have:
\[
\|D_{\phi_{R,R'}}\omega_R\|_{L^2}^2=\|\nabla \omega_R\|_{L^2}^2.
\]
Now, choosing for $\omega_R$, a normalized eigenfunction for the $k^{th}$-eigenvalue $C^k_R>0$ of the Dirichlet laplacian on $B(0,R),$ we get:
\[
\|D_{\phi_{R,R'}}\omega_R\|_{L^2}^2=C^k_R=\frac{C^k_1}{R^2}.
\]
Thus, the energy satisfies:
\begin{align*}
ls_k 	&\leq N\sqrt{\lambda^{k}({D_{\phi_{R,R'}}}^*D_{\phi_{R,R'}})}+\int_{\mathbb{R}^3}[\frac{|\nabla\phi_{R,R'}|^2}{2}+U(\phi_{R,R'})]dx\\
	&\leq\frac{N\sqrt{C^k_1}}{R}+ \frac{4\pi(m+\epsilon)^2}{6g^2}\frac{{R'}^2+RR'+R^2}{R'-R}+\frac{4\pi}{3}\left(\underset{r\in[-m/g,0]}{\sup} U(r)\right) R'^3.
\end{align*}

For now, we fix $R' =(1+\sqrt{3})R$, the point which minimizes $R'\mapsto\frac{{R'}^2+RR'+R^2}{R'-R}.$ If there exists $R>0$ such that :
\begin{eqnarray*}
	f(R) &&= \frac{N\sqrt{C^k_1}}{R} + \frac{4m^2(3+2\sqrt{3})\pi}{6g^2}R+\frac{4(1+\sqrt{3})^3\pi}{3}\left(\underset{r\in[-m/g,0]}{\sup} U(r)\right) R^3\\
	&&<Nm,
\end{eqnarray*}
we immediately get the result. Let us fix $R_0>0$ such that
\[
\frac{N\sqrt{C^k_1}}{R_0}<Nm.
\]
By hypothesis \eqref{condition:U0}, $U$ is continuous, $U(0)=0$ and we have
\[
	\underset{g\rightarrow +\infty}{\lim}~\underset{r\in[-m/g,0]}{\sup} U(r)=0.
\]
so that
\[
	\underset{g\rightarrow +\infty}{\lim}~\frac{4m^2(3+2\sqrt{3})\pi}{6g^2}R_0+\frac{4(1+\sqrt{3})^3\pi}{3}\left(\underset{r\in[-m/g,0]}{\sup} U(r)\right) R_0^3=0.
\]
Thus, there is $g_0>0$ such that if $g>g_0$ then
\[
\underset{R>0}{\inf}f(R)<Nm.
\]
\qed\end{proof}

A similar result holds for the bag approximation case.
\begin{lemma}\label{supersymmetry_estimationcavity}
Let $k\in\mathbb{N}\backslash\{0\}.$ Assume that $g\in(0,m).$ There is a constant $\delta>0$ such that if $a,~b<\delta$ then, 
\begin{equation*}
0<lc_k =\inf\{N\lambda^{k}_+(H_{-\chi_{\Omega}})+aP(\Omega)+b|\Omega|:\chi_{\Omega}\in F\}<Nm.
\end{equation*}
\end{lemma}

\begin{proof}\smartqed
For $0<R$, we choose $\chi_\Omega=\chi_{B(0,R)}$ and $\omega_k$ a normalized eigenfunction for the $k^{th}$-eigenvalue $C^k_R>0$ of the Dirichlet laplacian on $B(0,R).$ We get:
\[
0<lc_k\leq N\sqrt{\frac{C^k_1}{R^2} +(m-g)^2}+a4\pi R^2+b4/3\pi R^3,
\] 
and the result follows.
\qed\end{proof}

\subsection{The M.I.T. bag Dirac operator}

Just as in the previous case, the supersymmetry gives us a good frame to study the problem of the eigenvalues of the M.I.T. bag Dirac operator. We set $\Omega=B(0,R)$ with $R>0$,
	\[
		\begin{array}{lr}
			D = -i\boldsymbol{\sigma}.\nabla+im,		&\mathcal{D}(D) = \{\omega\in H^1_{sym}(\Omega,\mathbb{C}^2)~:~-\boldsymbol{\sigma}.n\omega=\omega~\text{on}~\partial\Omega\},\\
			D^* = -i\boldsymbol{\sigma}.\nabla-im,		&\mathcal{D}(D^*) = \{\omega\in H^1_{sym}(\Omega,\mathbb{C}^2)~:~\boldsymbol{\sigma}.n\omega=\omega~\text{on}~\partial\Omega\},
		\end{array}
	\]
	\[
		Q = TH_{0}T^{-1}=\left(\begin{array}{cc}
			0&D^*\\
			D&0
		\end{array}\right)
	\]
	and $\mathcal{D}(Q) = \mathcal{D}(D)\oplus\mathcal{D}(D^*)=T\mathcal{D}(H_0).$
	The following result implies Proposition \ref{prop:supersymmetryMIT1}.
\begin{proposition}\label{prop:supersymmetryMIT}
	The operator $(H_0,\mathcal{D}(H_0))$ is self-adjoint and there exists a non decreasing sequence of eigenvalues $(\lambda_n)_{n\geq1}\subset (m,+\infty)$ which tends to infinity such that:
	\[
		\sigma(H_0) = \{\dots,-\lambda_2,-\lambda_1\}\cup\{\lambda_1,\lambda_2,\dots\},
	\]
	 $D^*$ is the adjoint of $D$ and vice versa. We have for each $n$:
	\[
		\lambda^n_{MIT}(\Omega) = \lambda_n = \inf_{\substack{V\subset \mathcal{D}(D),\\\dim V=k}}\; \sup_{\substack{\omega\in V,\\\|\omega\|_{L^2}=1}}\|D\omega\|_{L^2(\Omega)}.
	\]
\end{proposition}
\begin{proof-sketch}\smartqed
The proof uses the spectral theory of self-adjoint compact operators and the ideas of the proof of Theorem \ref{supersymmetry_foldy}.
\qed\end{proof-sketch}

\section{Auxiliary results}
We study in this section, the dependance of the non-negative eigenvalue of $H_\phi$ on the field $\phi.$ This is an important point in this paper that will allow us to prove lower semi-continuity properties for the functionals involved in problems \eqref{variationalformulation} and \eqref{variationalformulationcavity}.   To prove Proposition \ref{auxiliary_convergence} below, we will need the two following lemmas.

\begin{lemma}\label{auxiliary_scalarfieldlimit}
Assume that  $(\phi_n)$ converges to $\phi_\infty$ strongly in $L^3(\mathbb{R}^3,\mathbb{R})$. Then, we have:
\[
\|D_{\phi_n}\omega\|^2_{L^2} \text{~converges to~} \|D_{\phi_\infty}\omega\|^2_{L^2}
\] 
locally uniformly in $\omega\in H^1(\mathbb{R}^3,\mathbb{C}^2)$ i.e., for every $R>0$:
\[
\sup\{|\|D_{\phi_n}\omega\|^2_{L^2}-\|D_{\phi_\infty}\omega\|^2_{L^2}|: \omega\in H^1(\mathbb{R}^3,\mathbb{C}^2), \|\omega\|_{H^1}\leq R\}\underset{n\rightarrow\infty}{\rightarrow}0.
\]
\end{lemma}

\begin{proof}\smartqed
We have
\begin{equation}\label{eq:identity1}
\|D_{\phi}\omega\|^2_{L^2}=\|\nabla \omega\|^2_{L^2}+\int_{\mathbb{R}^3}(m+g\phi)^2|\omega|^2dx-2g\mathcal{R}e\int_{\mathbb{R}^3}[\omega^*(\boldsymbol{\sigma}.\nabla\omega)\phi]dx.
\end{equation}
By H\"older's inequality, we get:
\begin{eqnarray*}
	\lefteqn{\left|\int_{\mathbb{R}^3}[\omega^*(\boldsymbol{\sigma}.\nabla\omega)\phi_n]-[\omega^*(\boldsymbol{\sigma}.\nabla\omega)\phi_{\infty}]dx\right|}\\
	&\leq \|\omega^*\boldsymbol{\sigma}.\nabla \omega\|_{L^{3/2}}\|\phi_n-{\phi_\infty}\|_{L^3}&\leq \|\omega\|_{L^6}\|\nabla \omega\|_{L^2}\|\phi_n-{\phi_\infty}\|_{L^3},
\end{eqnarray*}

\[
	\left|\int_{\mathbb{R}^3}|\omega|^2(\phi_n-{\phi_\infty})dx\right|\leq \|\omega\|_{L^3}^2\|\phi_n-{\phi_\infty}\|_{L^3},
\]
and
\[
	\left|\int_{\mathbb{R}^3}|\omega|^2(\phi_n^2-{\phi_\infty}^2)dx\right|\leq\|\omega\|_{L^6}^2\|\phi_n-{\phi_\infty}\|_{L^3}(\|\phi_n\|_{L^3}+\|\phi_\infty\|_{L^3}).
\]
The result follows.
\qed\end{proof}

\begin{lemma}\label{auxiliary_uniformcoecive}
Let $\epsilon>0.$ The functional $\omega\mapsto \|D_{\phi}\omega\|^2_{L^2}$ is coercive on 
\[\{\omega \in H^1(\mathbb{R}^3,\mathbb{C}^2): \|\omega\|_{L^2}=1\}\]
locally uniformly in $\phi\in L^3\cap L^{3+\epsilon}$ i.e., for every $R>0$, there is a $C>0$ such that
\[
\|D_{\phi}\omega\|^2_{L^2}\geq \|\nabla \omega\|_{L^2}^2-C
\]
for $\omega\in \{\omega \in H^1(\mathbb{R}^3,\mathbb{C}^2): \|\omega\|_{L^2}=1\}$ and $\phi$ such that $\|\phi\|_{L^3}+\|\phi\|_{L^{3+\epsilon}}\leq R.$
\end{lemma}

\begin{proof}\smartqed
By equality \eqref{eq:identity1} and H\"older's inequality, we have for $\omega\in \{\omega \in H^1(\mathbb{R}^3,\mathbb{C}^2): \|\omega\|_{L^2}=1\}$ that
\begin{eqnarray*}
	\lefteqn{\|D_{\phi}\omega\|^2_{L^2}\geq}\\
		& \|\nabla \omega\|_{L^2}^2+m^2&-2g\|\omega\|_{L^p}\|\nabla \omega\|_{L^2}\|\phi\|_{L^{3+\epsilon}}-2gm\|\omega\|_{L^2}\|\omega\|_{L^6}\|\phi\|_{L^3},
\end{eqnarray*}
where $p=\frac{6+2\epsilon}{1+\epsilon}\in(2,6).$ The Sobolev embedding $H^1\hookrightarrow L^6$ and the interpolation inequalities give us the result.
\qed\end{proof}

\begin{proposition}\label{auxiliary_convergence}
Let $\epsilon>0.$ Assume that a sequence $(\phi_n)$ converges to $\phi_\infty$ strongly in $L^3\cap L^{3+\epsilon}$. If for $k\geq 1,$ 
\begin{equation}\label{estimationeigenvalue}
\underset{n\in \mathbb{N}}{\sup}\lambda^k({D_{\phi_n}}^*D_{\phi_n})<m^2,
\end{equation}
then, up to a subsequence, there exist orthonormal families :
\[
(\omega^1_\infty,\dots, \omega^k_\infty) \text{~and~for all $n\in\mathbb{N},$~} (\omega^1_n,\dots, \omega^k_n)\text{~ in~} L^2(\mathbb{R}^3,\mathbb{C}^2)
\]
such that:
\begin{enumerate}
	\item \label{item0}$\omega^i_\infty,\omega^i_n\in H^1(\mathbb{R}^3,\mathbb{C}^2)$ for all $n$,
	\item\label{item1} $\|D_{\phi_n}\omega^i_n\|^2_{L^2}=\lambda^i({D_{\phi_n}}^*D_{\phi_n})$ for all $n$,
	\item\label{item2} $\|D_{\phi_\infty}\omega^i_\infty\|_{L^2}^2=\lambda^i({D_{\phi_\infty}}^*D_{\phi_\infty})$,
	\item\label{item3} $\underset{n\rightarrow\infty}{\lim}\lambda^i({D_{\phi_n}}^*D_{\phi_n})=\lambda^i({D_{\phi_\infty}}^*D_{\phi_\infty}),$
	\item\label{item4} $\underset{n\rightarrow\infty}{\lim}\omega^i_n=\omega^i_\infty$ in $H^1,$
\end{enumerate}
for all $1\leq i\leq k.$
\end{proposition}

This proposition is true in the symmetric case too.

\begin{proof}\smartqed
We will prove this by induction on $k.$ Let $k\geq 1$. Assume that inequality \eqref{estimationeigenvalue} is true for $k$. Then, if $k>1$, we have that
\[
	\underset{n\in \mathbb{N}}{\sup}\lambda^{k-1}({D_{\phi_n}}^*D_{\phi_n})\leq \underset{n\in \mathbb{N}}{\sup}\lambda^k({D_{\phi_n}}^*D_{\phi_n})<m^2.
\]
Assume that the proposition is true for $k-1.$ We get by induction hypothesis, that there exist orthonormal families :
\[
(\omega^1_\infty,\dots, \omega^{k-1}_\infty) \text{~and~for all $n\in\mathbb{N},$~} (\omega^1_n,\dots, \omega^{k-1}_n)\text{~ in~} L^2(\mathbb{R}^3,\mathbb{C}^2)
\] 
such that the properties \ref{item0},\dots,\ref{item4} are true for all $1\leq i\leq k-1$. If $k=1$, these families are chosen empty. By theorem \ref{supersymmetry_foldy} and inequality \eqref{estimationeigenvalue}, there exist for all $n,$ $\omega^k_n\in H^1$ such that $(\omega^1_n,\dots ,\omega^k_n)$ is orthonormal in $L^2$ and point \eqref{item1} is satisfied for all $1\leq i\leq k $. It turns out that $(\omega^k_n)$ is a bounded sequence of $H^1$ by Lemma \ref{auxiliary_uniformcoecive} and  
\[
\underset{n\rightarrow\infty}{\lim}\|D_{\phi_n}\omega^k_n\|^2_{L^2}-\|D_{\phi_\infty}\omega^k_n\|^2_{L^2}=0,
\]
by Lemma \ref{auxiliary_scalarfieldlimit} because $(\phi_n)$ converges to $\phi_\infty$ in $L^3\cap L^{3+\epsilon}$.

Let us denote by $P^{k-1}_n$ and $P^{k-1}_\infty$ the orthogonal projector on 
\[\text{span}(\omega^1_n,\dots, \omega^{k-1}_n)^\perp\]
 and \[E_k:=\text{span}(\omega^1_\infty,\dots, \omega^{k-1}_\infty)^\perp.\]
 Then, we have by point \eqref{item4} of the induction hypothesis:
\[
\underset{n\rightarrow\infty}{\lim}P^{k-1}_n \omega^k_n-P^{k-1}_\infty \omega^k_n=0\text{~in~} H^1,
\]
hence, the family $(\omega^1_\infty,\dots, \omega^{k-1}_\infty,\omega^k_n)$ is free for $n$ large enough. We also get
\[
\underset{n\rightarrow\infty}{\lim}\|D_{\phi_\infty}\omega^k_n\|_{L^2}^2-\|D_{\phi_\infty}P^{k-1}_\infty \omega^k_n\|_{L^2}^2=0.
\]
Thus, we obtain that :
\[
\lambda^k({D_{\phi_\infty}}^*D_{\phi_\infty})\leq\underset{n\rightarrow\infty}{\lim}{\inf}\lambda^k({D_{\phi_n}}^*D_{\phi_n})<m^2.
\]
Therefore, there exists $\tilde{\omega}^k_\infty\in \text{span}(\omega^1_\infty,\dots, \omega^{k-1}_\infty)^\perp$ such that $\|\tilde{\omega}^k_\infty\|_{L^2}=1$ and
\[
\|D_{\phi_\infty}\tilde{\omega}^k_\infty\|_{L^2}^2=\lambda^k({D_{\phi_\infty}}^*D_{\phi_\infty}),
\]
so, by the same arguments, we get:
\[
\begin{array}{ll}
\|D_{\phi_\infty}\tilde{\omega}^k_\infty\|_{L^2}^2&=\|D_{\phi_\infty}P^{k-1}_n\tilde{\omega}^k_\infty\|_{L^2}^2+o(1)\\
&=\|D_{\phi_n}P^{k-1}_n\tilde{\omega}^k_\infty\|_{L^2}^2+o(1),
\end{array}
\]
and 
\[
\lambda^k({D_{\phi_\infty}}^*D_{\phi_\infty})\geq\underset{n\rightarrow\infty}{\lim}{\sup}\lambda^k({D_{\phi_n}}^*D_{\phi_n}).
\]
This gives us point \eqref{item3} for all $1\leq i\leq k$. Moreover, $(\frac{P^{k-1}_\infty \omega^k_n}{\|P^{k-1}_\infty \omega^k_n\|_{L^2}})$ is a minimizing sequence of :
\[
\inf\{\|D_{\phi_\infty}\omega\|_{L^2}^2:\omega\in \text{span}(\omega^1_\infty,\dots, \omega^{k-1}_\infty)^\perp:\|\omega\|_{L^2}=1\}=\lambda^k({D_{\phi_\infty}}^*D_{\phi_\infty}).
\]
By  theorem \ref{supersymmetry_foldy}, we get that
\[
\lambda^k({D_{\phi_\infty}}^*D_{\phi_\infty})=\inf \sigma({{D_{\phi_\infty}}^*D_{\phi_\infty}}_{|E_k}) <m^2=\inf\sigma_{ess}({{D_{\phi_\infty}}^*D_{\phi_\infty}}_{|E_k})
\]
where ${{D_{\phi_\infty}}^*D_{\phi_\infty}}_{|E_k}$ is the operator  ${{D_{\phi_\infty}}^*D_{\phi_\infty}}$ restricted to the domain
  \[
  \mathcal{D}({{D_{\phi_\infty}}^*D_{\phi_\infty}})\cap E_k.
  \]
   Thus, up to a subsequence, $(\frac{P^{k-1}_\infty \omega^k_n}{\|P^{k-1}_\infty \omega^k_n\|_{L^2}})$  and $(\omega^k_n)$ converge in $H^1$ and we get points \eqref{item2} and \eqref{item4} for all $1\leq i\leq k$.
\qed\end{proof}

\section{The symmetric case}

\subsection{Pre-compactness results}

\subsubsection{The soliton bag}
We show now the existence of a minimizer of problem \eqref{variationalformulation}.
\begin{lemma}\label{lem:compactnesssymmetricsoliton}
Let $1\leq k_1\leq\dots \leq k_N\leq K.$ Assume that:
\[
0<ls_K = {\inf}\left\{N\lambda^{K}_+(H_{\phi})+\int_{\mathbb{R}^3}[\frac{|\nabla\phi|^2}{2}+U(\phi)]dx: \phi\in H^1_{rad}\right\}<Nm.
\]
Then, there exists a minimizer $\phi\in H^1_{rad}$ of problem \eqref{variationalformulation}.
\end{lemma}
The pre-compactness of a minimizing sequence is obtained thanks to the compactness of the embeddings of $H^1_{rad}(\mathbb{R}^3)$ into $L^p(\mathbb{R}^3)$ for all $p\in(2,6)$ proven by Strauss \cite{Strauss} and generalized by Lions \cite{Lions1982}. 
\begin{proof}\smartqed
There exists a minimizing sequence $(\phi_n)\subset H^1_{rad}(\mathbb{R}^3,\mathbb{R})$ of problem \eqref{variationalformulation} such that:
\[
\underset{n\in \mathbb{N}}{\sup}\;N\lambda^{K}_+(H_{\phi_n})+\int_{\mathbb{R}^3} \left[\frac{|\nabla\phi_n|^2}{2}+U(\phi_n)\right]dx<Nm,
\]
so, by the non-negativeness of $U$ and Theorem \ref{supersymmetry_foldy}, for all $i\in\{1,\dots,N\}:$
\[
\underset{n\in \mathbb{N}}{\sup}\;\lambda^{k_i}({D_{\phi_n}}^*D_{\phi_n})<m^2.
\]
Because of \eqref{condition:U2}, $(\phi_n)$ is a bounded sequence of $H^1_{rad}(\mathbb{R}^3,\mathbb{R})$. By the compactness properties of the radial Sobolev spaces due to Strauss \cite{Strauss} and Lions \cite{Lions1982}, there exists $\phi_\infty \in H^1_{rad}(\mathbb{R}^3,\mathbb{R})$ such that, up to a subsequence, we have:
\[
\phi_n\rightarrow \phi_\infty  \text{~weakly  in $H^1$, a.a, strongly in $L^p$ for $2<p<6.$~}
\]
Thus, by Proposition \ref{auxiliary_convergence}, up to another subsequence, we have:
\[
	\underset{k\rightarrow\infty}{\lim}\lambda^i({D_{\phi_{n_k}}}^*D_{\phi_{n_k}})=\lambda^i({D_{\phi_\infty}}^*D_{\phi_\infty}),
\]
for all $i\in\{1,\dots,N\}$ so,
\[
	\begin{array}{rrl}
		\underset{k\rightarrow\infty}{\lim\inf}	&\sum_{\substack{i=1}}^{\substack{N}}\lambda^{k_i}_+(H_{\phi_{n_k}})&+\int_{\mathbb{R}^3} \left[\frac{|\nabla\phi_{n_k}|^2}{2}+U(\phi_{n_k})\right]dx\\
									&= \sum_{\substack{i=1}}^{\substack{N}}\lambda^{k_i}_+(H_{\phi_\infty})&+\underset{k\rightarrow\infty}{\lim\inf}\text{~}\int_{\mathbb{R}^3} \left[\frac{|\nabla\phi_{n_k}|^2}{2}+U(\phi_{n_k})\right]dx\\
									&\geq \sum_{\substack{i=1}}^{\substack{N}}\lambda^{k_i}_+(H_{\phi_\infty})&+\int_{\mathbb{R}^3} \left[\frac{|\nabla\phi_{\infty}|^2}{2}+U(\phi_{\infty})\right]dx.
	\end{array}
\]
This ensures that $\phi_{\infty}$ is a minimum of problem \eqref{variationalformulation} and that $(\phi_{n_k})$ tends to $\phi_{\infty}$ strongly in $H^1$.
\qed\end{proof}
\subsubsection{The bag approximation}
As in the previous part, we prove the existence of a minimizer of problem \eqref{variationalformulationcavity}.
\begin{lemma}
Let $1\leq k_1\leq\dots \leq k_N\leq K.$ Assume that:
\[
0<lc_K = {\inf}\left\{N\lambda^{K}_+(H_{-\chi_\Omega})+aP(\Omega)+b|\Omega|: \chi_{\Omega}\in BV_{rad}\right\}<Nm.
\]
Then, there exists a minimizer $\chi_{\Omega}\in BV_{rad}$ of problem \eqref{variationalformulationcavity}.
\end{lemma}

The arguments are very similar to the ones of Lemma \ref{lem:compactnesssymmetricsoliton} and we give here only a sketch of proof to stress the differences.

\begin{proof-sketch}\smartqed 
\[
\chi_\Omega\in BV\mapsto \sum_{\substack{i=1}}^{\substack{N}}\lambda^{k_i}_+(H_{-\chi_{\Omega}})+aP(\Omega)+b|\Omega|
\]
is lower semi-continuous for the topology of $L^1$ thanks to the lower semicontinuity of
\[
\phi\in BV\mapsto |\nabla\phi|(\mathbb{R}^3)
\]
in the topology of $L^1$ and Proposition \ref{auxiliary_convergence}. For the reader's convenience, we give in the appendix the proof of the compactness of some embeddings in the $BV$ setting similar to the ones of Strauss \cite{Strauss} and Lions \cite{Lions1982}. The pre-compactness of a minimizing sequence follows then from Proposition \ref{prop:BVcompactness}.
\qed\end{proof-sketch}

\subsection{Euler-Lagrange equations}
We get in the last section the existence of a $\phi$ which minimizes \eqref{variationalformulation} or \eqref{variationalformulationcavity} for given $1\leq k_1\leq\dots \leq k_N\leq K$. Thus, $H_{\phi}$ has at least $k_N$ eigenvalues in $(0,m)$ associated with normalized eigenvectors $(\psi_{1},\dots, \psi_N)$. It remains to shows that $(\psi_1,\dots,\psi_N,\phi)$ satisfies the Euler-Lagrange equations \eqref{equation} or \eqref{equation:cavity}.

\begin{lemma}\label{eulerlagrange}
Let $1\leq k_1\leq\dots \leq k_N\leq K.$ The functions $\psi_{1},\dots, \psi_N,\phi$ obtained by minimization of \eqref{variationalformulation} satisfy the Euler-Lagrange equations \eqref{equation} of the soliton bag model. The same is true for the bag approximation.
\end{lemma}

\begin{proof}\smartqed
We give the proof only in the soliton case. For the bag approximation, the proof follows with the same argument in the setting of set derivation (see \cite{ambrosio2000functions,Giusti1984,Pierre2005} for more details). By Lemma \ref{simplicity}, we get that $0\notin\sigma(H_\phi)$ and every eigenvalue is simple. Let $\lambda(H_{\phi})\in(0,m)$ be an eigenvalue of $H_{\phi}$ and $\phi'\in H^1_{rad},$ by Kato-Rellich theorem for the perturbation of the point spectra \cite[Theorem 12.8]{Reed1978}, we have two analytic functions in a neighborhood $\mathcal{U}$ of $0$:
\[
t\mapsto\lambda(H_{\phi+t\phi'}) \text{~and~} t\mapsto \psi_{\phi+t\phi'},
\]
where for each $t\in\mathcal{U},$ $\lambda(H_{\phi+t\phi'})$ is a simple eigenvalue of $H_{\phi+t\phi'}$ and $\psi_{\phi+t\phi'}$ is an  associated normalized eigenvector. Thus, we have, for each $t\in\mathcal{U}$:
\[
\lambda(H_{\phi+t\phi'}) = (\psi_{\phi+t\phi'},H_{\phi+t\phi'}\psi_{\phi+t\phi'}),
\]
so,
\begin{align*}
\frac{d}{dt}\lambda(H_{\phi+t\phi'}) &= 2\mathcal{R}e(\frac{d}{dt}\psi_{\phi+t\phi'},H_{\phi+t\phi'}\psi_{\phi+t\phi'})+g(\psi_{\phi+t\phi'},\beta \phi'\psi_{\phi+t\phi'})\\
&=\lambda(H_{\phi+t\phi'})2\mathcal{R}e(\frac{d}{dt}\psi_{\phi+t\phi'},\psi_{\phi+t\phi'})+g(\psi_{\phi+t\phi'},\beta \phi'\psi_{\phi+t\phi'})\\
&=g(\psi_{\phi+t\phi'},\beta \phi'\psi_{\phi+t\phi'}),\\
\end{align*}
because $\|\psi_{\phi+t\phi'}\|_{L^2}=1$ for all $t\in\mathcal{U}.$  This ensures that:
\[
-\Delta \phi+U'(\phi)+\sum_{\substack{i=1}}^{\substack{N}}g\psi_i^*\beta\psi_i=0,
\]
and we get Lemma \ref{eulerlagrange}.
\qed\end{proof}

This ends the proofs of Theorems \ref{maintheorem} and \ref{maintheorem3}

\section{The non-symmetric case}
\subsection{Pre-compactness results}
\subsubsection{The soliton case}
We will now focus on the existence of a ground state solution of equations \eqref{equation} in the non-symmetric case. The concentration compactness method allows us to deal with the lack of compactness of $H^1(\mathbb{R}^3)$ thanks to the so-called concentration-compactness inequality. Nevertheless, the classical one \cite{Lions1984-1} is not valid yet. In the following, we will introduce a different concentration-compactness inequality to overcome this problem. We denote:
\begin{eqnarray*}
	\lefteqn{\mathcal{E}(\phi_1,\phi_2,t) = N\left(t\lambda^1_+(H_{\phi_{1}})^2+(1-t)\lambda^1_+(H_{\phi_{2}})^2\right)^{1/2}}\\
	&&+\int_{\mathbb{R}^3}\left(\frac{|\nabla \phi_1|^2}{2}+U(\phi_1)\right)dx+\int_{\mathbb{R}^3}\left(\frac{|\nabla \phi_2|^2}{2}+U(\phi_2)\right)dx,
\end{eqnarray*}
for $\phi_1,\phi_2\in H^1(\mathbb{R}^3,\mathbb{R})$ and $t\in [0,1];$
\begin{eqnarray*}
	I(t)=\inf\{\mathcal{E}(\phi_1,\phi_2,t): \phi_1,\phi_2\in H^1(\mathbb{R}^3,\mathbb{R})\}.
\end{eqnarray*}
The following lemma is related to the concentration compactness inequality. 
\begin{lemma}\label{binding}
$I$ is concave, $I(0)=I(1),$
\[I(t)\leq Nm,~\text{for all}~t\in[0,1]\]
 and the concentration-compactness inequality 
\[I(t)\geq I(1),~\text{for all}~t\in[0,1]\]
is satisfied.
\end{lemma}
\begin{proof}\smartqed
$I$ is concave as an infimum of concave functions. The remaining follows noticing that:
\[
	\mathcal{E}(0,0,t)= Nm.
\]
\qed\end{proof}
We can now prove the existence of a minimizer of \eqref{variationalformulation} thanks to the concentration compactness method and Lemma \ref{binding}.
\begin{lemma}\label{lemma:existencesolitonnonsymmetric}
Let us assume that $I(1)<Nm,$ then every minimizing sequence $(\phi_n)$ of 
\[
	I(1)=\inf\{N\lambda^{1}_+(H_\phi)+\int_{\mathbb{R}^3}\left(\frac{|\nabla \phi|^2}{2}+U(\phi)\right)dx: ~\phi\in H^1(\mathbb{R}^3,\mathbb{R})\}
\]
converges in $H^1$ to a minimum of problem \eqref{variationalformulation}, up to translation and extraction.
\end{lemma}

\begin{proof}\smartqed
Let $(\phi_n)$ be a minimizing sequence such that:
\[
\sup_{\substack{n\in \mathbb{N}}}\left(N\lambda^{1}_+(H_{\phi_n})+\int_{\mathbb{R}^3}\left(\frac{|\nabla \phi_n|^2}{2}+U(\phi_n)\right)dx\right)<Nm.
\]
$(\phi_n)$ is a bounded sequence in $H^1$ because of \eqref{condition:U2} and 
\[\sup_{\substack{n\in\mathbb{N}}}\lambda^1({D_{\phi_n}}^*D_{\phi_n})<m^2.\] 

We will now apply the concentration compactness method  to get the result. We follow the presentation of Lewin \cite{Lewin2010} based on \cite{Lieb} (see also \cite{Lions1984-1,Struwe2008}). Let us assume first that the sequence vanishes, then:
\[
\phi_n \rightarrow 0~\text{strongly in}~L^p~\text{for}~p\in(2,6),
\] 
and Proposition \ref{auxiliary_convergence} leads us to a contradiction. So, there exists a subsequence $(n_k)$, a sequence $(x_{n_k})\subset \mathbb{R}^3$ and $\phi_0 \in H^1\backslash \{0\}$ such that:
\[
\phi_{n_k}(~.-x_{n_k}) \rightharpoonup \phi _0~\text{weakly in}~H^1.
\]
We define $(\omega_{n_k})$ a sequence of $H^1$ such that $\|\omega_{n_k}\|_{L^2}=1$ and 
\[
	\|D_{\phi_{n_k}}\omega_{n_k}\|_{L^2} = \lambda^{1}_+(H_{\phi_{n_k}}).
\]
Up to extraction, there is $\omega_0\in H^1$ such that:
\[
\omega_{n_k}(~.-x_{n_k}) \rightharpoonup \omega_0~\text{weakly in}~H^1.
\]
Let $(R_k)$ be an increasing sequence of $\mathbb{R}^+$ such that $\lim_{\substack{k\rightarrow \infty}}R_k = \infty,$ then, up to a subsequence, there exists 
\[
(\phi_{1,k}), (\phi_{2,k}), (\omega_{1,k}), (\omega_{2,k}) \subset H^1,
\]
such that:
\begin{enumerate}
\item $\begin{array}{c} \|\omega_{n_k}-\omega_{1,k}-\omega_{2,k}\|_{H^1}\rightarrow 0,\\\|\phi_{n_k}-\phi_{1,k}-\phi_{2,k}\|_{H^1}\rightarrow 0,\end{array}$
\item $\begin{array}{c} \omega_{1,k}(~.-x_{n_k})\rightarrow \omega_0\\\phi_{1,k}(~.-x_{n_k})\rightarrow\phi_0\end{array}~\text{weakly in}~H^1,~\text{strongly in}~L^p~\text{for}~p\in[2,6),$
\item $\begin{array}{l}supp(\phi_{1,k})\cup supp(\omega_{1,k})\subset B(x_{n_k},R_k),\\supp(\phi_{2,k})\cup supp(\omega_{2,k})\subset \mathbb{R}^3\backslash B(x_{n_k},2R_k).\end{array}$
\end{enumerate}

We get:
\begin{eqnarray*}
	\lefteqn{I(1)=\underset{k\rightarrow+\infty}{\lim\inf}~ N\|D_{\phi_{n_k}}\omega_{n_k}\|_{L^2} +\int_{\mathbb{R}^3}\left(\frac{|\nabla \phi_{n_k}|^2}{2}+U(\phi_{n_k})\right)dx}\\
	&&\geq \underset{k\rightarrow+\infty}{\lim\inf}~N\left\{\|D_{\phi_{1,k}}\omega_{1,k}\|_{L^2}^2+\|D_{\phi_{2,k}}\omega_{2,k}\|_{L^2}^2\right\}^{1/2}\\
	&&\dots+\int_{\mathbb{R}^3}\left(\frac{|\nabla \phi_{1,k}|^2}{2}+U(\phi_{1,k})\right)dx+\int_{\mathbb{R}^3}\left(\frac{|\nabla \phi_{2,k}|^2}{2}+U(\phi_{2,k})\right)dx\\
	&&\geq \underset{k\rightarrow+\infty}{\lim\inf}~N\left\{\|\omega_0\|_{L^2}^2\lambda^1_+(H_{\phi_{0}})^2+(1-\|\omega_0\|_{L^2}^2)\lambda^1_+(H_{\phi_{2,k}})^2\right\}^{1/2}\\
	&&\dots+\int_{\mathbb{R}^3}\left(\frac{|\nabla \phi_{0}|^2}{2}+U(\phi_{0})\right)dx+\int_{\mathbb{R}^3}\left(\frac{|\nabla \phi_{2,k}|^2}{2}+U(\phi_{2,k})\right)dx\\
	&& \geq I(\|\omega_0\|_{L^2}^2).
\end{eqnarray*}

Since $\phi_0\ne 0,$ $\omega_0$ has to be non zero.  Assume now that $\|\omega_0\|_{L^2}\in(0,1).$ Lemma \ref{binding} ensures that :
\[
I(t)=I(1),~\text{for all }~t\in [0,1].
\] 
We must have:
\[
	\underset{k\rightarrow+\infty}{\lim}~ \lambda^1_+(H_{\phi_{2,k}}) = \lambda^1_+(H_{\phi_0}).
\]
If not, assume for instance that there exists a another subsequence such that:
\[
	\underset{k\rightarrow+\infty}{\lim}~ \lambda^1_+(H_{\phi_{2,k}}) > \lambda^1_+(H_{\phi_0}),
\]
then,
\begin{eqnarray*}
	\lefteqn{I(1)= \underset{k\rightarrow+\infty}{\lim\inf}~ \mathcal{E}(\phi_0,\phi_{2,k},\|\omega_0\|^2_{L^2})}\\
	&&>  \underset{k\rightarrow+\infty}{\lim\inf}~ \mathcal{E}(\phi_0,\phi_{2,k},1)\\
	&&\geq I(1).
\end{eqnarray*}
This is impossible. The same argument leads to a contradiction with:
\[
	\underset{k\rightarrow+\infty}{\lim}~ \lambda^1_+(H_{\phi_{2,k}}) < \lambda^1_+(H_{\phi_0}).
\]
Thus, we get:
\begin{eqnarray*}
	\lefteqn{I(1)=\underset{k\rightarrow+\infty}{\lim\inf}~ N\lambda^1_+(H_{\phi_{0}})}\\
	 &&\dots+\int_{\mathbb{R}^3}\left(\frac{|\nabla \phi_{0}|^2}{2}+U(\phi_{0})\right)dx+\int_{\mathbb{R}^3}\left(\frac{|\nabla \phi_{2,k}|^2}{2}+U(\phi_{2,k})\right)dx
\end{eqnarray*}
and
\[
	\underset{k\rightarrow+\infty}{\lim\inf}~\int_{\mathbb{R}^3}\left(\frac{|\nabla \phi_{2,k}|^2}{2}+U(\phi_{2,k})\right)dx=0.
\]
By Proposition \ref{auxiliary_convergence}, we get the contradiction: 
\[
	m=\underset{k\rightarrow+\infty}{\lim}~ \lambda^1_+(H_{\phi_{2,k}}) = \lambda^1_+(H_{\phi_0}).
\]
Thus, we have $\|\omega_0\|_{L^2}=1$ and 
\[
	\underset{k\rightarrow+\infty}{\lim\inf}~\int_{\mathbb{R}^3}\left(\frac{|\nabla \phi_{2,k}|^2}{2}+U(\phi_{2,k})\right)dx=0.
\]
The result follows.
\qed\end{proof}
\subsubsection{The bag approximation}
We follow exactly the same ideas. Let us introduce some notations:
\begin{eqnarray*}
	\lefteqn{\mathcal{F}(\chi_{\Omega_1},\chi_{\Omega_2}t) =N\left(t\lambda^1_+(H_{-\chi_{\Omega_1}})^2+(1-t)\lambda^1_+(H_{-\chi_{\Omega_2}})^2\right)^{1/2}}\\
	 &&\dots+aP(\Omega_1)+b|\Omega_1|+aP(\Omega_2)+b|\Omega_2|,
\end{eqnarray*}
for $\chi_{\Omega_1},\chi_{\Omega_2}\in BV(\mathbb{R}^3,\mathbb{R})$ and $t\in [0,1];$
\begin{eqnarray*}
	J(t)=\inf\{\mathcal{F}(\chi_{\Omega_1},\chi_{\Omega_2}t): \chi_{\Omega_1},\chi_{\Omega_2}\in BV(\mathbb{R}^3,\mathbb{R})\}.
\end{eqnarray*}
\begin{lemma}\label{binding2}
$J$ is concave, $J(0)=J(1),$
\[J(t)\leq Nm,~\text{for all}~t\in[0,1]\]
 and the concentration-compactness inequality 
\[J(t)\geq J(1),~\text{for all}~t\in[0,1]\]
is satisfied.
\end{lemma}
\begin{proof-sketch} \smartqed
	The proof is similar to the one of Lemma \ref{binding}.
\qed\end{proof-sketch}
\begin{lemma}\label{lemma:bagapproxcc}
Let us assume that $J(1)<Nm,$ then for every minimizing sequence $(\chi_{\Omega_n})$ of 
\[
	J(1)=\inf\{N\lambda^{1}_+(H_{-\chi_\Omega})+aP(\Omega)+b|\Omega|: ~\chi_\Omega\in BV\},
\]
converges strongly in $BV$ to a minimum of \eqref{variationalformulationcavity} up to translation and extraction.
\end{lemma}
\begin{proof-sketch} \smartqed
The proof is similar to the one of the soliton case. For the reader's convenience, we give in the appendix the straightforward adaptation of the presentation of the concentration compactness method of Lewin \cite{Lewin2010} to the $BV$ setting. 
\qed\end{proof-sketch}
\subsection{Euler-Lagrange equations}
As in the symmetric case, it remains to show that the minimizer satisfies the Euler-Lagrange equations.
\begin{lemma}\label{eulerlagrangesoliton}
Let $\phi\in H^1(\mathbb{R}^3,\mathbb{R})$ and $\omega\in H^1(\mathbb{R}^3,\mathbb{C}^2)$ be such that $\|\omega\|_{L^2}=1$,
\[
I(1) = N\|D_{\phi}\omega\|_{L^2} +\int_{\mathbb{R}^3}\left(\frac{|\nabla\phi|^2}{2}+U(\phi)\right)dx,
\]
and $\|D_{\phi}\omega\|_{L^2}>0$ then,
\[
	\begin{array}{ll}
		-\Delta\phi+U'(\phi)+Ng\psi^*\beta\psi=0,
	\end{array}
\]
where 
\[
\psi = (U_{FW}T)^*\left[\begin{array}{c}\omega\\0\end{array}\right]
\]
is an normalized eigenvector of $H_{\phi}$ associated with the smallest positive eigenvalue $\lambda = \|D_{\phi}\omega\|_{L^2}$.
\end{lemma}
\begin{remark}
If $\|D_{\phi}\omega\|_{L^2}=0$ then $\phi$ satisfies an Euler-Lagrange inequation. 
\end{remark}
\begin{proof}\smartqed
We have:
\[
	\begin{array}{ll}
		-\Delta\phi+U'(\phi)+\frac{Ng}{\lambda}\mathcal{R}e\left[\omega^*(-iD_{\phi}\omega)\right]=0,
	\end{array}
\]
and
\begin{align*}
	(\psi,\beta\psi) 	&= \left(U_{FW}^*\left[\begin{array}{c}\omega\\0\end{array}\right]\right)^*T\beta T^*\left(U_{FW}^*\left[\begin{array}{c}\omega\\0\end{array}\right]\right) \\
				&= \frac{1}{2}\left[\begin{array}{c}\omega\\S\omega\end{array}\right]^*\left(\begin{array}{cc}0&-i\\i&0\end{array}\right)\left[\begin{array}{c}\omega\\S\omega\end{array}\right]\\
				&=\frac{1}{\lambda}\mathcal{R}e\left[\omega^*(-iD_{\phi}\omega)\right]. \\
\end{align*}
\qed\end{proof}

The next lemma shows that the minimizers of \eqref{variationalformulationcavity} also satisfy Euler-Lagrange equations.
\begin{lemma}
Assume that $g\in(0,m)$. Let $\chi_\Omega\in BV(\mathbb{R}^3,\mathbb{R})$ and  $\omega\in H^1(\mathbb{R}^3,\mathbb{C}^2)$ be such that $\|\omega\|_{L^2}=1$ and
\[
J(1) = N\|D_{-\chi_\Omega}\omega\|_{L^2} +aP(\Omega)+b|\Omega|,
\]
then,
\[
	\begin{array}{ll}
		a\mathcal{H}_\Omega+b-Ng\psi^*\beta\psi=0,~\text{on}~\partial \Omega
	\end{array}
\]
where 
\[
\psi = (U_{FW}T)^*\left[\begin{array}{c}\omega\\0\end{array}\right]
\]
is an normalized eigenvector of $H_{-\chi_\Omega}$ associated with the smallest positive eigenvalue $\lambda = \|D_{-\chi_\Omega}\omega\|_{L^2}$.
\end{lemma}
\begin{proof-sketch}\smartqed
We have:
\[
	a\mathcal{H}+b-\frac{Ng}{\lambda}\mathcal{R}e\left[\omega^*(-iD_{\phi}\omega)\right]=0,~\text{on}~\partial^* \Omega.
\]
The arguments of the proof of the previous lemma give the result.
\qed\end{proof-sketch}

This ends the proofs of theorems  \ref{maintheorem2} and \ref{maintheorem4}.

\section{Gamma convergence results}
We give here the proof of Proposition \ref{gammalimit} based on \cite{Modica1987,Braides1998,braides2006,Sternberg}:
\begin{proof}\smartqed
	Let $(\epsilon_n)$ be a decreasing sequence converging to $0$ and $(\phi_n)$ be such that:
	\[
		\left\{\begin{array}{ll}
			\underset{n\rightarrow+\infty}{\lim}\phi_n=\phi&\text{in}~L^p~\text{for all}~ p\in[2,\frac{3(q+2)}{4}],\\
			\underset{n\rightarrow+\infty}{\lim}E_{\epsilon_n}(\phi_n)&\text{exists and is finite.}
		\end{array}\right.
	\]
	Up to extraction, we can assume that $(\phi_n)$ tends to $\phi$ almost everywhere. $(\phi_n)\subset H^1$ is a bounded sequence in $L^2$ and
	\[
		\underset{n\rightarrow+\infty}{\lim\inf}\int_{\mathbb{R}^3}W(\phi_n)dx=\int_{\mathbb{R}^3}W(\phi)dx=0.
	\]
	So, there exists a subset $\Omega$ of $\mathbb{R}^3$ such that $\phi=-\chi_\Omega$ and $|\Omega|=\|\phi\|_{L^2}^2<+\infty.$  Moreover, we have for all $n$ by Cauchy-Schwarz inequality:
	\[
		\int_{\mathbb{R}^3}\left(\epsilon_n|\nabla \phi_n|^2+W(\phi_n)/\epsilon_n\right)dx\geq\int_{\mathbb{R}^3}2|\nabla\phi_n|\sqrt{W(\phi_n)}dx=|\nabla (\mathcal{W}\circ\phi_n)|(\mathbb{R}^3),
	\]
	where $\mathcal{W}(t)=2\int_0^t\sqrt{W(s)}ds$ and $|\nabla w|(\mathcal{A})$ denotes the variation of $w\in L^1$ on the Borel set $\mathcal{A}$. Since, there is $C>0$ such that:
	\[
		\mathcal{W}(t)\leq C(|t|^2+|t|^{\frac{q+2}{2}})~\forall t,
	\]
	$(\mathcal{W}\circ\phi_n)$ is bounded in $BV,$ converges to $\mathcal{W}\circ\phi$ in $L^p$ for all $p\in[1,3/2].$ Thus, we get:
	\[
		\left\{\begin{array}{ll}
			\mathcal{W}\circ\phi = a\chi_{\Omega}\in BV,~\phi=-\chi_\Omega\in BV,\\
			\underset{n\rightarrow+\infty}{\lim\inf}|\nabla (\mathcal{W}\circ\phi_n)|(\mathbb{R}^3) \geq |\nabla (\mathcal{W}\circ\phi)|(\mathbb{R}^3)=aP(\Omega),
		\end{array}\right.
	\]
	so,
	\[
		\underset{n\rightarrow+\infty}{\lim\inf}~E_{\epsilon_n}(\phi_n) \geq E_{0}(\phi)
	\]		
	and
	\[
		\left(\Gamma-\underset{\epsilon\rightarrow0}{\lim\inf}~E_\epsilon\right)(\phi)\geq E_{0}(\phi).
	\]
	It remains to construct recovering sequences. For $R>0$ and every $\tilde{\Omega}\subset \subset B(0,R)$ such that $\chi_{\tilde{\Omega}}\in BV,$ Sternberg \cite{Sternberg} constructs  a sequence $(\phi_\epsilon)\subset H^1_0(B(0,R))$ such that:
	\[
		\left\{\begin{array}{l}
			(\phi_\epsilon)~\text{converges to}~ -\chi_{\tilde{\Omega}}~\text{in}~L^1(B(0,R)),\\
			\|\phi_\epsilon\|_{L^\infty}\leq 1~\text{for all}~ \epsilon>0,\\
			\underset{\epsilon\rightarrow 0}{\lim\sup}\int_{B(0,R)}\left(\epsilon|\nabla \phi_\epsilon|^2+W(\phi_\epsilon)/\epsilon\right)dx= aP(\tilde{\Omega}).
		\end{array}\right.
	\] 
	thus, 
	\[
		\left(\Gamma-\underset{\epsilon\rightarrow 0}{\lim\sup}~E_{\epsilon}\right)(-\chi_{\tilde{\Omega}})\leq E_0(-\chi_{\tilde{\Omega}}).
	\]
	For every $\chi_\Omega\in BV,$ we have 
	\begin{eqnarray*}
		(\chi_{\Omega\cap B(0,R)})_{R>0}\subset BV \text{~ tends to ~}\chi_\Omega \text{~in~} L^p \text{~for all~} p\in[1,+\infty),\\
		(|\nabla \chi_{\Omega\cap B(0,R)}|(\mathbb{R}^3))\text{~ tends to ~}|\nabla \chi_{\Omega}|(\mathbb{R}^3),
	\end{eqnarray*}
	so, 
	\[
		\underset{R\rightarrow +\infty}{\lim}~E_0(-\chi_{\Omega\cap B(0,R)})= E_0(-\chi_{\Omega}). 
	\]
	Since the $\Gamma-$limit-sup is lower semi-continuous, we obtain:
	\[
		\begin{array}{ll}
			\left(\Gamma-\underset{\epsilon\rightarrow 0}{\lim\sup}~E_{\epsilon}\right)(-\chi_\Omega)	&\leq\underset{R\rightarrow+\infty}{\lim\inf}~\left(\Gamma-\underset{\epsilon\rightarrow 0}{\lim\sup}~E_{\epsilon}(-\chi_{\Omega\cap B(0,R)})\right)\\
																			&\leq \underset{R\rightarrow+\infty}{\lim\inf}~E_{0}(-\chi_{\Omega\cap B(0,R)})\\
																			&\leq E_0(-\chi_\Omega).
		\end{array}
	\]
\qed\end{proof}
The following lemmas are part of the proof of Theorem \ref{gammalimit2}. We introduce for $\epsilon>0$:
		\[
			\mathcal{Z}_\epsilon = \{\phi\in L^2(\mathbb{R}^3,\mathbb{R})\cap L^{\frac{3(q+2)}{4}}(\mathbb{R}^3,\mathbb{R}):~\int_{\mathbb{R}^3}\frac{W(\phi)}{\epsilon}\leq Nm\},
		\]
		\[
			G_\epsilon(\phi) = \left\{%
			\begin{array}{ll}
				|\nabla \mathcal{W}\circ\phi|(\mathbb{R}^3)+b\|\phi\|^2_{L^2} 	&\text{if $\phi\in \mathcal{Z}_\epsilon$,}\\
				+\infty													&\text{otherwise,}
			\end{array}
			\right.
		\]
		and for $t\in[0,1], \phi_1\in\{-\chi_\Omega\in BV\},~\phi_2\in\mathcal{Z}_\epsilon,$
\begin{eqnarray*}
	\mathcal{F}_{\epsilon}(\phi_1,\phi_2,t)=N\left\{t\lambda^1_+(H_{\phi_1})^2+(1-t)\lambda^1_+(H_{\phi_2})^2\right\}^{1/2}+E_0(\phi_1)+G_\epsilon(\phi_2)\\
	J_\epsilon(t) = \inf\{\mathcal{F}_{\epsilon}(\phi_1,\phi_2,t):~\phi_1\in\{-\chi_\Omega\in BV\},~\phi_2\in\mathcal{Z}_\epsilon\}
\end{eqnarray*}

\begin{lemma}\label{gammabinding}
	We have for all $\epsilon>0,$ that $J_\epsilon$ is concave, continuous and
	\begin{eqnarray}
		\label{eq2:gammabinding}	0\leq J_\epsilon(t)\leq Nm
	\end{eqnarray}
	for all $t\in[0,1].$
	There exists a concave function $J_0$ such that $(J_{\epsilon})$ tends to $J_0$  pointwise in $[0,1]$ as $\epsilon$ tends to $0$ and
	\[
		0\leq J_0(0)\leq J_0(t),
	\]
	for all $t\in[0,1].$
\end{lemma}
\begin{proof}\smartqed 
	The same argument as in Lemma \ref{binding} gives us inequality \eqref{eq2:gammabinding}. $J_\epsilon$ is concave and continuous as an infimum of concave and continuous functions. $(J_\epsilon)$ is a non-increasing sequence since $(\mathcal{Z}_\epsilon)$ is non-decreasing sequence of sets. Hence, $(J_\epsilon)$ converges point-wise to a concave function $J_0$ in $[0,1]$ as $\epsilon$ tends to $0$. The remaining follows immediately.
\qed\end{proof}

The core of the proof of Theorem \ref{gammalimit2} is given by the following lemma. We use here the concentration compactness method and the $\Gamma$-convergence theory.
\begin{lemma}\label{lemma:gamma2}
We have:
\[
	J_0(0) = J_0(1).
\]
If $J_0(0)<Nm$ and for all $n,$ there is $\phi_n\in \mathcal{Z}_{\epsilon_n}$ such that:
\begin{eqnarray*}
	\underset{n\rightarrow+\infty}{\lim}N\lambda^1_+(H_{\phi_n})+G_{\epsilon_n}(\phi_n) = J_0(0),
\end{eqnarray*}
where $(\epsilon_n)$ is a sequence which tends to $0$, then, up to a subsequence, up to translation,
	\[
		\left\{\begin{array}{l}
			\mathcal{W}\circ\phi_n\rightarrow \mathcal{W}\circ(-\chi_\Omega)~\text{strictly in}~BV\\
			\phi_n\rightarrow-\chi_\Omega~\text{strongly in}~L^p~\text{for}~p\in[2,\frac{3(q+2)}{4}]
		\end{array}\right.
	\]
	where $\chi_\Omega\in BV.$

\end{lemma}
\begin{proof}\smartqed
	If $J_0(0)=Nm,$ then, $J_0(0)= J_0(1)$ by Lemma \ref{gammabinding}. Thus, we can assume that $J_0(0)<Nm.$ Let $\phi_n\in\mathcal{Z}_{\epsilon_n}$ be such that:
	\begin{eqnarray*}
		\underset{n\rightarrow+\infty}{\lim}~N\lambda^1_+(H_{\phi_n})+G_{\epsilon_n}(\phi_n) = J_0(0).
	\end{eqnarray*}
	We can assume that:
	\begin{eqnarray*}
		\underset{n}{\sup}~N\lambda^1_+(H_{\phi_n})+G_{\epsilon_n}(\phi_n) <Nm.
	\end{eqnarray*}
	As in the proof of \ref{gammalimit}, $(\phi_n)$ is uniformly bounded in $L^2$ and 
	\[
		\underset{n}{\sup}~|\nabla \mathcal{W}\circ\phi_n|(\mathbb{R}^3)<Nm.
	\]
	By Sobolev embedding, $(\mathcal{W}\circ\phi_n)$ is a bounded sequence of $L^{3/2}$. Since there is a positive constant $c>0$ such that:
	\[
		\mathcal{W}(t)\geq c|t|^{\frac{q+2}{2}},
	\] 
	for all $t$, $(\phi_n)$ is bounded in $L^{\frac{3(q+2)}{4}}$ and by the interpolation inequalities in $L^p$ for all $p\in[2,\frac{3(q+2)}{4}].$ We get that $(\mathcal{W}\circ\phi_n)$ is uniformly bounded in $BV.$ 	
	
	Let us assume now that this sequence vanishes. Then, $(\mathcal{W}\circ\phi_n)$ tends to $0$ in $L^p$ for all $p\in(1,3/2)$, so $(\phi_n)$ tends to $0$ in $L^p$ for all $p\in(2,\frac{3(q+2)}{4}).$ Proposition \ref{auxiliary_convergence} contradicts
	\[
		J_0(0)<Nm.
	\]
	Thus, up to a subsequence, there exist  $(x_n)\subset \mathbb{R}^3$ and $\tilde{\mathcal{W}}\in BV\backslash\{0\}$ such that $(\mathcal{W}\circ\phi_n(~.~-x_n))$ tends to  $\tilde{\mathcal{W}}$ in $L^1_{loc}.$ Since for all $n$, $\phi_n$ belongs to $\mathcal{Z}_{\epsilon_n},$ there exists $\chi_\Omega$ such that $(\phi_n(~.~-x_n))$ tends to $-\chi_\Omega$ almost everywhere, up to another subsequence and $\tilde{\mathcal{W}} = \mathcal{W}\circ(-\chi_\Omega)=-a\chi_\Omega\in BV$. 
	
	For all $n,$ there exists moreover $\omega_n\in H^1$ such that $\|\omega_n\|_{L^2}=1$ and 
	\[
		\|D_{\phi_n(~.~-x_n)}\omega_n\|_{L^2}=\lambda^1_+(H_{\phi_n}).
	\]
 	 By Lemma \ref{auxiliary_uniformcoecive}, $(\omega_n)$ is uniformly bounded in $H^1.$ Up to a subsequence, $(\omega_n)$ tends to $\omega\in H^1$ in $H^1$ weakly.
	
	Let $0<R_k $ be a sequence such that $(R_k)$ tends to $+\infty$. Then, by concentration compactness, there exist:
	\[
		(\omega_{1,n}), (\omega_{2,n}) \subset H^1,
	\]
 	such that, up to a subsequence,
	\[
		\left\{\begin{array}{l}
			\|\omega_{n}-\omega_{1,n}-\omega_{2,n}\|_{H^1}\rightarrow 0,\\
			\omega_{1,n}\rightarrow \omega~\text{weakly in}~H^1,~\text{strongly in}~L^p~\text{for}~p\in[2,6),\\
			supp(\omega_{1,n})\subset B(0,R_n),\\ 
			supp(\omega_{2,n})\subset \mathbb{R}^3\backslash B(0,2R_n)
		\end{array}\right.
	\]
	and
	\begin{equation}\label{eqn:ccchamp}
		\left\{\begin{array}{l}
			\chi_{B(0,R_n)}(\mathcal{W}\circ\phi_n(~.~-x_n))~\text{ tends to }~ \mathcal{W}\circ(-\chi_\Omega)~ \text{in}~ L^p~\text{ if}~ 1\leq p<3/2,\\ 
			\int_{R_n<|x|<2R_n}(|\mathcal{W}\circ\phi_n(~.~-x_n)|+|\nabla \mathcal{W}\circ\phi_n(~.~-x_n)|)dx~\text{ tends to}~ 0.
		\end{array}\right.
	\end{equation}
	  
	We localize now the $\phi$ field. Let us define for all $n$
	\[
		\phi_{1,n} = \phi_n\chi_{B(x_n,\frac{3R_n}{2})}, ~\phi_{2,n} = \phi_n\chi_{B(x_n,\frac{3R_n}{2})^c}\in \mathcal{Z}_{\epsilon_n}.
	\]
	Then, following the same notation of theorem $3.84$ of \cite{ambrosio2000functions}, we have that $\phi_{i,n}$ belongs to $BV(\mathbb{R}^3,\mathbb{R})$ for $i\in\{1,2\}$ and 
	\begin{eqnarray*}
		|\nabla \mathcal{W}\circ \phi_{1,n}|(\mathbb{R}^3) = |\nabla \mathcal{W}\circ \phi_{n}|\left(B(x_n,\frac{3R_n}{2})\right)+\int_{\partial B(x_n,\frac{3R_n}{2})}|(\mathcal{W}\circ \phi_{n})^+|ds,
	\end{eqnarray*}
	\begin{eqnarray*}
		|\nabla \mathcal{W}\circ \phi_{2,n}|(\mathbb{R}^3) = |\nabla \mathcal{W}\circ \phi_{n}|\left(B(x_n,\frac{3R_n}{2})^c\right)+\int_{\partial B(x_n,\frac{3R_n}{2})}|(\mathcal{W}\circ \phi_{n})^-|ds.
	\end{eqnarray*}
	Theorem $3.86$ of \cite{ambrosio2000functions} ensures moreover that there exists a constant $c>0$ such that for all $w\in BV(\mathcal{A}):$
	\[
		\int_{\partial B(0,\frac{3}{2})}|w^\pm|ds\leq c\left(\|w\|_{L^1(\mathcal{A})}+|\nabla w|(\mathcal{A})\right),
	\]
	where $\mathcal{A} = B(0,2)\backslash B(0,1).$
	By a rescaling argument, we get that for all $R>1,$ for all $w\in BV(\mathcal{A}_R):$
	\begin{eqnarray*}
		\int_{\partial B(0,\frac{3R}{2})}|w^\pm|ds&\leq c\left(\|w\|_{L^1(\mathcal{A}_R)}/R+|\nabla w|(\mathcal{A}_R)\right)\\
		&\leq c\left(\|w\|_{L^1(\mathcal{A}_R)}+|\nabla w|(\mathcal{A}_R)\right),
	\end{eqnarray*}
	where $\mathcal{A}_R = B(0,2R)\backslash B(0,R).$
	
	We obtain thanks to equations \eqref{eqn:ccchamp}:
	\begin{eqnarray*}
		J_0(0)= \lefteqn{\underset{n\rightarrow+\infty}{\lim}~N\|D_{\phi_n(~.~-x_n)}\omega_n\|_{L^2}	+G_{\epsilon_n}(\phi_n)}\\
		&&\geq \underset{n\rightarrow+\infty}{\lim\inf}~N\left\{\|D_{\phi_{1,n}(~.~-x_n)}\omega_n^1\|_{L^2(\mathbb{R}^3)}^2+\|D_{\phi_{2,n}(~.~-x_n)}\omega_n^2\|_{L^2(\mathbb{R}^3)}^2\right\}^{1/2}\\
		&&+G_{\epsilon_n}(\phi_{1,n})+G_{\epsilon_n}(\phi_{2,n})\\
		&&\geq \underset{n\rightarrow+\infty}{\lim\inf}~N\left\{\|\omega\|^2_{L^2}\lambda^1_+(H_{-\chi_\Omega})^2 +(1-\|\omega\|^2_{L^2})\lambda^1_+(H_{\phi_{2,n}})^2\right\}^{1/2}\\
		&&+E_{0}(-\chi_\Omega)+G_{\epsilon_n}(\phi_{2,n})\\
		&& \geq J_0(\|\omega\|^2_{L^2}).
	\end{eqnarray*}

	This imposes $\|\omega\|_{L^2}>0,$ otherwise, $J_0(0)\geq c+J_0(0)$ with $c>0.$ If $\|\omega\|_{L^2}\in (0,1),$ then we have by Lemma \ref{gammabinding}:
	\[
		J_0(t) = J_0(0),
	\]
	for all $t\in[0,1].$ As in the proof of Lemma \ref{lemma:existencesolitonnonsymmetric}, we have:
	\[
		\underset{n\rightarrow+\infty}{\lim}~\lambda^1_+(H_{\phi_{2,n}}) = \lambda^1_+(H_{-\chi_\Omega})
	\]
	and we must have:
	\[
		\underset{n\rightarrow+\infty}{\lim\inf}~G_{\epsilon_n}(\phi_{2,n})=0,
	\]
	so, we get the contradiction:
	\[
		\underset{n\rightarrow+\infty}{\lim}~\lambda^1_+(H_{\phi_{2,n}})=m=\lambda^1_+(H_{-\chi_\Omega})<m.
	\]

	Thus, we obtain that $\|\omega\|_{L^2}=1$ and
	\[
		\underset{n\rightarrow+\infty}{\lim\inf}~G_{\epsilon_n}(\phi_{2,n})=0,
	\]
	so that,
	\[
		\left\{\begin{array}{l}
			\mathcal{W}\circ\left[\phi_n(~.~-x_n)\right]\rightarrow \mathcal{W}\circ(-\chi_\Omega)~\text{strictly in}~BV\\
			\phi_n(~.~-x_n)\rightarrow-\chi_\Omega~\text{strongly in}~L^p~\text{for}~p\in[2,\frac{3(q+2)}{4}]
		\end{array}\right.
	\]	
	and
	\[
		J_0(0)=J_0(1) = N\lambda^1_+(H_{-\chi_\Omega})+E_{0}(-\chi_\Omega).
	\]
\qed\end{proof}

Let us write the proof of Theorem \ref{gammalimit2} which follows from Proposition \ref{gammalimit} and the previous lemmas.
\begin{proof}\smartqed
	The first part of the theorem follows from Proposition \ref{auxiliary_convergence} and the fact that the $\Gamma$-convergence remains true if we add continuous functions.
	
	We assume next that:
	\[
		J_0(1)<Nm.
	\]
	Lemma \ref{lemma:bagapproxcc} ensures that there exists $-\chi_\Omega\in BV$ such that:
	\[
		N\lambda^1_+(H_{-\chi_\Omega}) + E_0(-\chi_\Omega) = J_0(1).
	\]
	By Proposition \ref{gammalimit}, there is a sequence $(\phi_\epsilon)\subset H^1$ such that:
	\[\left\{\begin{array}{l}
		\underset{\epsilon\rightarrow 0 }{\lim\sup} ~E_\epsilon(\phi_\epsilon)\leq E_0(-\chi_\Omega),\\
		\phi_\epsilon\rightarrow -\chi_\Omega,~\text{in} ~L^2\cap L^{\frac{3(q+2)}{4}}.
	\end{array}\right.\]
	Thus, we get:	
	\[
		\underset{\epsilon\rightarrow 0 }{\lim\sup}~ l^\epsilon_s\leq \underset{\epsilon\rightarrow 0 }{\lim\sup}~N\lambda^1_+(H_{\phi_\epsilon})+E_\epsilon(\phi_\epsilon)\leq J_0(1)<Nm.
	\]
	There exists $\epsilon_0>0$ such that for all $0<\epsilon<\epsilon_0,$
	\[
		l_s^\epsilon< Nm,
	\]
	and by Lemma \ref{lemma:existencesolitonnonsymmetric}, there is $\phi_\epsilon\in H^1$ such that:
	\[
	 	\mathcal{E}_\epsilon(\phi_\epsilon)=l_s^\epsilon.
	\]

	We have:
	\begin{eqnarray*}
		\lefteqn{J_0(1)\geq \underset{\epsilon\rightarrow 0 }{\lim\sup}~l^\epsilon_s,}\\
		&&\geq \underset{\epsilon\rightarrow 0 }{\lim\sup}~N\lambda^1_+(H_{\phi_\epsilon})+G_\epsilon(\phi_\epsilon),\\
		&&\geq \underset{\epsilon\rightarrow 0 }{\lim\sup}~J_\epsilon(0)=J_0(1),
	\end{eqnarray*}
	and Lemma \ref{lemma:gamma2} concludes the proof.
\qed\end{proof}

The proof of Corollary \ref{cor:solitongroundgeneral} follows immediately from Theorem \ref{gammalimit2} and Proposition \ref{auxiliary_convergence}.
We write now the proof of Theorem \ref{gammalimit3}.
\begin{proof}\smartqed
	Just as in the proof of Theorem \ref{gammalimit2}, the $\Gamma-$ convergence follows from Proposition \ref{auxiliary_convergence} and 
	\[
		\underset{\epsilon\rightarrow0}{\lim\sup}~l_s^\epsilon(k_1,\dots,k_N)\leq\underset{\epsilon\rightarrow0}{\lim\sup}~l_s^\epsilon(K,\dots,K)\leq l_c(K,\dots,K)<Nm.
	\]
	Lemma \ref{lemma:existencesolitonnonsymmetric} ensures that there exists $\epsilon_0>0$ such that for all $0<\epsilon<\epsilon_0,$ problem \eqref{lsepsilonex} has a minimum $\phi_\epsilon\in BV_{rad}.$ We get that $(\mathcal{W}\circ\phi_\epsilon)$ is bounded in $BV$ and $(\phi_\epsilon)$ in $L^2.$ So by Proposition \ref{prop:BVcompactness}, there exists a subsequence $(\epsilon_n)$ and $\phi=-\chi_\Omega\in BV$ such that:
	\[
		\left\{\begin{array}{ll}
			\mathcal{W}\circ\phi_n\rightarrow \mathcal{W}\circ\phi~\text{strongly in}~L^p~\text{for all}~p\in(1,3/2)~\text{and}~a.a,\\
			\phi_n\rightarrow \phi~\text{strongly in}~L^p~\text{for all}~p\in(2,\frac{3(q+2)}{4}).
		\end{array}\right.
	\]
	By Proposition \ref{auxiliary_convergence}, we have
	\[
		\begin{array}{lll}
			\underset{n\rightarrow}{\lim\inf}	&\mathcal{E}_{\epsilon_n,k_1,\dots,k_N}(\phi_n)\\
										&\geq\sum_{i=1}^N\lambda^{k_i}_+(H_\phi)			&+\underset{n\rightarrow+\infty}{\lim\inf}~E_{\epsilon_n}(\phi_n)\\
										&\geq \mathcal{E}_{0,k_1,\dots,k_N}(\phi).
		\end{array}
	\]
	Thus, we get the conclusion of the theorem.
	
\qed\end{proof}

\section{The M.I.T. bag limit}
We give here the proofs of Theorem \ref{theo:MITlimit} and Proposition \ref{prop:MITgroundstate}.

\begin{lemma}\label{lemma:estimationMIT}
	Let $\chi_\Omega\in BV(\mathbb{R}^3,\mathbb{R})$ and $0<m<M.$ We have for $\omega\in H^1(\mathbb{R}^3,\mathbb{C}^2)$:
	\begin{eqnarray*}
		\lefteqn{\|-\boldsymbol{\sigma}.\nabla \omega+(m\chi_\Omega+M\chi_{\Omega^c})\omega\|^2_{L^2(\mathbb{R}^3)} = }\\
		&&\|\boldsymbol{\sigma}.\nabla \omega\|^2_{L^2(\mathbb{R}^3)}+m^2\|\omega\|^2_{L^2(\Omega)}+M^2\|\omega\|^2_{L^2(\Omega^c)}+(M-m)\int_{\partial\Omega}\omega^*(\boldsymbol{\sigma}.n)\omega dz\\
	\end{eqnarray*}
	There is $C>0$ such that:
	\begin{eqnarray*}
		\lefteqn{C\|-\boldsymbol{\sigma}.\nabla \omega+(m\chi_\Omega+M\chi_{\Omega^c})\omega\|^2_{L^2(\mathbb{R}^3)} \geq}\\
		&& \frac{1}{M}\|\nabla \omega\|^2_{L^2(\mathbb{R}^3)}+\|\omega\|^2_{L^2(\Omega)}+M\|\omega\|^2_{L^2(\Omega^c)}+\|\boldsymbol{\sigma}.\nabla \omega\|^2_{L^2(\Omega)},\\
	\end{eqnarray*}
	$C$ depends neither on $\Omega$ nor on $\omega.$
\end{lemma}
\begin{proof}\smartqed
	Let $c_1, c_2>0,$ we have: 
	\[
		\begin{array}{ll}
		\|-c_1\boldsymbol{\sigma}.\nabla \omega+c_2\omega\|^2_{L^2(\Omega^c)} &= c_1^2\|\boldsymbol{\sigma}.\nabla \omega\|^2_{L^2(\Omega^c)}+c^2_2\|\omega\|^2_{L^2(\Omega^c)}\\
		&\dots-c_1c_2\int_{\partial\Omega}\omega^*(\boldsymbol{\sigma}.n)\omega dz.\\
		\end{array}
	\]
	So, we get:
	\begin{eqnarray*}
		\lefteqn{2\|-\boldsymbol{\sigma}.\nabla \omega+(m\chi_\Omega+M\chi_{\Omega^c})\omega\|^2_{L^2(\mathbb{R}^3)}}\\
		 &&\geq \frac{2Mm-m^2}{M^2}\|\nabla \omega\|^2_{L^2(\mathbb{R}^3)}+m^2\|\omega\|^2_{L^2(\Omega)}+2Mm\|\omega\|^2_{L^2(\Omega^c)}+\|\boldsymbol{\sigma}.\nabla \omega\|^2_{L^2(\Omega)}.\\
	\end{eqnarray*}
\qed\end{proof}

\begin{lemma}\label{lemma:estimationsigmanabla}
	For any $C>0,$ there exists a constant $c_0>0$ such that if $\omega$ belongs to $H^1_{sym}(\mathbb{R}^3,\mathbb{C}^2)$ and satisfies:
	\[
		\begin{array}{ll}
		C &\geq \frac{1}{M}\|\nabla \omega\|^2_{L^2(\mathbb{R}^3)}+M\|\omega\|^2_{L^2(B(0,R)^c)}+\|\boldsymbol{\sigma}.\nabla \omega\|^2_{L^2(B(0,R))}+\|\omega\|^2_{L^2(B(0,R))},\\
		\end{array}
	\]
	for $M>m$ and $R>0,$ then,
	\[
		\|\omega\|_{L^{3}(\mathbb{R}^3)}\leq c_0.
	\]
	$c_0$ does not depend on $M,$ $R$ or $\omega.$
\end{lemma}
\begin{proof}\smartqed
	Let $\omega \in H^1_{sym}(\mathbb{R}^3,\mathbb{C}^2),$ 
	\[
	\omega(x) = v(r)\left(\begin{array}{c}
				1\\
				0
	\end{array}\right)+u(r)\left(\begin{array}{c}
				cos(\theta)\\
				sin(\theta)e^{i\varphi}
	\end{array}\right)
	\]
	where $(r,\theta,\varphi)$ are the spherical coordinates of $x$. We have:
	\[
		\|\nabla \omega\|_{L^2(B(0,R))}^2 = 4\pi\int_0^R\left[|u'(r)|^2+\frac{2|u(r)|^2}{r^2}+|v'(r)|^2\right]r^2dr 
	\]
	and
	\[
		\|\boldsymbol{\sigma}.\nabla \omega\|_{L^2(B(0,R))}^2 = 4\pi\int_0^R\left[\left|u'(r)+\frac{2u(r)}{r}\right|^2+|v'(r)|^2\right]r^2dr.
	\]
	Since,
	\begin{eqnarray*}
		\lefteqn{\int_0^R|u'(s)+\frac{2u(s)}{s}|^2s^2ds 	}\\
										&&= \int_0^R\left[|u'(s)|^2s^2+4|u(s)|^2+4u(s)u'(s)s\right]ds\\
										&&= \int_0^R\left[|u'(s)|^2s^2+4|u(s)|^2-2|u(s)|^2\right]ds+2R|u(R)|^2\\
										&&\geq  \int_0^R\left[|u'(s)|^2s^2+2|u(s)|^2\right]ds
	\end{eqnarray*}
	we get:
	\[
		\|\boldsymbol{\sigma}.\nabla \omega\|_{L^2(B(0,R))}^2\geq \|\nabla \omega\|_{L^2(B(0,R))}^2.
	\]
	Let us remark that this inequality is wrong when the domain is an annulus and $u(r)=1/r^2$. 
	
	Hence, we obtain:
	\[
		\begin{array}{ll}
			\int_{B(0,R)}|\nabla |\omega|^2|dx\leq \|\boldsymbol{\sigma}.\nabla \omega\|_{L^2(B(0,R))}^2+\|\omega\|_{L^2(B(0,R))}^2\leq C
		\end{array}	
	\]
	and 
	\[
		\begin{array}{ll}
			\int_{B(0,R)^c}|\nabla |\omega|^2|dx\leq \frac{1}{M}\|\nabla \omega\|_{L^2(B(0,R)^c)}^2+M\|\omega\|_{L^2(B(0,R)^c)}^2\leq C
		\end{array}	
	\]	
	By Sobolev injection, we get the result.	
\qed\end{proof}

We are now able to give the proof of Theorem \ref{theo:MITlimit}.
\begin{proof}\smartqed
Lemma \ref{supersymmetry_estimationcavity} and the arguments of the proof of Theorem \ref{maintheorem3} ensure that there exist $C_0,~ n_0>0,$ for $n\geq n_0,$ a radius $R_n>0$ minimizing
	\[
		\inf\left\{N\lambda^{1}_+(H^n_{B(0,R)})+aP(B(0,R))+b|B(0,R)|:R>0\right\},
	\]
	a function 
	\[
		\psi_n(x) = \left(\begin{array}{c}v_n(r)\left(\begin{array}{c}1\\0\end{array}\right)\\iu_n(r)\left(\begin{array}{c}\cos\theta\\\sin\theta e^{i\varphi}\end{array}\right)\end{array}\right) \in H^{1}_{sym}(\mathbb{R}^3,\mathbb{C}^4)
	\]
	satisfying:
	\[
		\left\{\begin{array}{ll}
			H^n_{B(0,R_n)}\psi_n 	&= \lambda^{1}_+(H^n_{B(0,R)})\psi_n\\
			\|\psi_n\|_{L^2}			&=1
		\end{array}\right.
	\]
	where $H^n_{\Omega}=-i\alpha.\nabla+\beta(m\chi_\Omega+M_n\chi_{\Omega^c})$ and 
	\[
		N\lambda^{1}_+(H^n_{B(0,R_n)})+aP(B(0,R_n))+b|B(0,R)|\leq C_0.
	\]
	 $(R_n)$ is bounded, so there exists a subsequence $(n_k)$ and $R\geq 0$ such that 
		\[
			R_{n_k}\rightarrow R.
		\]
	We claim that $R>0$. Indeed, there is for all $n$, a function $\omega_n\in H^1_{sym}(\mathbb{R}^3,\mathbb{C}^4)$ such that:
		\[
			\left\{
				\begin{array}{ll}
					\|-\boldsymbol{\sigma}.\nabla \omega_n+(m\chi_{B(0,R_n)}+M\chi_{B(0,R_n)^c})\omega_n\|^2_{L^2(\mathbb{R}^3)} = \lambda^{1}_+(H^n_{B(0,R_n)})^2\\
					\|\omega_n\|_{L^2}=1.
				\end{array}
			\right.
		\]
		Lemmas \ref{lemma:estimationMIT} and \ref{lemma:estimationsigmanabla} ensure that $(\omega_{n_k})$ is a bounded sequence in $L^{3}(\mathbb{R}^3),$ such that:
		\[
			\|\omega_{n_k}\|_{L^2(B(0,R_{n_k})^c)}\rightarrow 0,
		\]
		so, $R$ has to be positif.

	We denote $\lambda_n:= \lambda^{1}_+(H^n_{B(0,R)}).$ $(\lambda_n)$ is a bounded sequence of $(m,+\infty)$ so, up to a subsequence, we can assume that it converges to $\lambda\in [m,+\infty).$ $(u_n,v_n)$ satisfies
	\[
		\left\{
		\begin{array}{ll}
			u_n'(r)+\frac{2u_n(r)}{r}	&=-(m-\lambda_n)v_n(r)\\
			v_n'(r)				&=-(m+\lambda_n)u_n(r)
		\end{array}
		\right.
	\]
	for $r\in(0,R)$ and 
	\[
		\|\psi_n\|_{L^2(B(0,R_n)^c)} \leq \frac{C_0}{M_n}. 
	\]
	We get:
	\[
		\left\{
		\begin{array}{ll}
			r^2u_n''(r)+2ru_n'(r)-(r^2(m^2-\lambda_n^2)+2)u_n(r)	&=0\\
			r^2v_n''(r)+2rv_n'(r)-r^2(m^2-\lambda_n^2)v_n(r)		&=0.
		\end{array}
		\right.
	\]
	So, $u_n$ and $v_n$ are spherical Bessel functions on $(0,R_n),$ they have to be of the first kind to belong to $L^2.$ Thus, $(u_n,v_n)$
 is proportional to	
 	\[
		r\mapsto 
		\left(
		\begin{array}{c}
			\frac{sin(\sqrt{\lambda_n^2-m^2}r)}{(\sqrt{\lambda_n^2-m^2}r)^2}-\frac{cos(\sqrt{\lambda_n^2-m^2}r)}{\sqrt{\lambda_n^2-m^2}r}\\
			\sqrt{\frac{\lambda_n^2+m^2}{\lambda_n^2-m^2}}\frac{sin(\sqrt{\lambda_n^2-m^2}r)}{\sqrt{\lambda_n^2-m^2}r}
		\end{array}
		\right)
	\]
	on $(0,R_n).$ We also get that $(u_n,v_n)$
 is proportional to	
 	\[
		r\mapsto 
		\left(
		\begin{array}{c}
			\frac{1+\sqrt{M_n^2-\lambda_n^2}r}{(\sqrt{M_n^2-\lambda_n^2}r)^2}exp(-\sqrt{M_n^2-\lambda_n^2}r)\\
			\sqrt{\frac{M_n^2+\lambda_n^2}{M_n^2-\lambda_n^2}}\frac{exp(-\sqrt{M_n^2-\lambda_n^2}r)}{\sqrt{M_n^2-\lambda_n^2}r}
		\end{array}
		\right)
	\]
	on $[R_n,+\infty)$, so, 
	\[
		\frac{u_n(R_n)}{v_n(R_n)}\rightarrow 1.
	\]
	We finally get that $(u_n,v_n)$ converges uniformly to a function $(u,v)$ on 
	\[[0,+\infty)\backslash(R-\epsilon,R+\epsilon)\]
	for any $\epsilon>0$ where
	\[
		\left\{
		\begin{array}{ll}
			u'(r)+\frac{2u(r)}{r}	&=-(m-\lambda)v(r)\\
			v'(r)				&=-(m+\lambda)u(r)
		\end{array}
		\right.
	\] 
	on $(0,R],$ $u(r)=v(r)=0$ on $(R,+\infty)$ and 
	\[
		u(R) = \underset{r\rightarrow R^-}{\lim}u(r)=\underset{r\rightarrow R^-}{\lim}v(r)=v(R).
	\]
	It remains to prove that $\lambda=  \lambda^1_{MIT}(B(0,R)).$  Let $R>0$ be the radius of the ball $B(0,R)$ that minimizes 
	\[
			\inf\left\{N\lambda^1_{MIT}(B(0,r))+aP(B(0,r))+b|B(0,r)|:~r>0\right\},
	\]
	and $\omega\in H^1_{sym}(B(0,R),\mathbb{C}^2)$ be an normalized function satisfying
	\[
		\lambda^1_{MIT}(B(0,R))=\|D\omega\|_{L^2}
	\]
	and 
	\[
		\omega(x) = c\left(
		\begin{array}{c}
			cos(\theta)-1\\
			sin(\theta)e^{i\varphi}
		\end{array}
		\right)
	\]
	where $(r,\theta,\varphi)$ are the spherical coordinates of $|x|=R.$  We set
	\[
		\omega_n(x) = c_n\left\{\begin{array}{ll}
		c\left(\begin{array}{c}
			cos(\theta)-1\\
			sin(\theta)e^{i\varphi}
		\end{array}
		\right)exp(-M_n(r-R))~&\text{for}~x\in B(0,R)^c\\
		w(x)~&\text{for}~x\in B(0,R)
		\end{array}\right.
	\]
	where $\omega_n$ is normalized by $c_n$.
	We get:
	\[
		\|-\boldsymbol{\sigma}.\nabla \omega_n+M_n\omega_n\|^2_{L^2(B(0,R)^c)}\rightarrow 0
	\]
	and the result follows.
\qed\end{proof}

Let us now prove Proposition \ref{prop:MITgroundstate}.
\begin{proof}\smartqed
	Let $R$ be a fixed positive constant. We recall that
	\[
		\lambda^1_{MIT}(B(0,R))=\inf\{\|D\omega\|_{L^2(B(0,R))}:~\omega\in H^1_{sym}(B(0,R),\mathbb{C}^2),~\|\omega\|_{L^2(B(0,R))}=1\}.
	\]
	 Let $\omega\in H^1_{sym}(B(0,R),\mathbb{C}^2)$.
	We get that:
	\begin{eqnarray*}
		\|D\omega\|_{L^2(B(0,R))}^2 = \|\sigma.\nabla\omega\|_{L^2(B(0,R))}^2+m^2+m\|\omega\|_{L^2(\partial B(0,R))}^2
	\end{eqnarray*}

	A scaling argument shows that $R\mapsto\lambda^1_{MIT}(B(0,R))$ is a convex decreasing function such that
	\[
		\left\{\begin{array}{ll}
			\underset{R\rightarrow+\infty}{\lim}\lambda^1_{MIT}(B(0,R))=m,\\
			\underset{R\rightarrow0}{\lim}\lambda^1_{MIT}(B(0,R))=+\infty.
		\end{array}\right.
	\]
	We get that $R\mapsto N\lambda^1_{MIT}(B(0,R))+aP(B(0,R))+b|B(0,R)|$ is a strictly convex and coercive function. Hence, the minimum exists and is unique.
\qed\end{proof}

\appendix
\section{A compactness result for bounded variation functions with symmetry}
Whereas the embedding of $H^1(\mathbb{R}^N)$ in $L^p(\mathbb{R}^N)$ for $N>2$ and $p\in[2,\frac{2N}{N-2}]$ is not compact, Strauss \cite{Strauss} showed that the restrictions of these embeddings to radial functions are compact for $p\in(2,\frac{2N}{N-2}).$ This result has been generalized by Lions \cite{Lions1982} to other Sobolev spaces. 

The adaptation of the proofs of Lions to the $BV$ setting is straightforward and is given here for the reader's convenience. We denote by $BV_{rad}(\mathbb{R}^N)$ the subset of $BV(\mathbb{R}^N)$ of radial functions where $N\in\mathbb{N}\backslash\{0\}.$ The following lemma gives a control of the decay at infinity of the radial $BV$ functions.

\begin{lemma}\label{lemma:inequalityforradialfunction}
Let $N>1,~u\in BV_{rad}(\mathbb{R}^N),$ then we have:
\[
|u(x)|\leq \left\{|\nabla u|(\mathbb{R}^N)|x|^{-(N-1)}\right\}~\text{a.a.} ~x\in \mathbb{R}^N.
\]
\end{lemma}
\begin{proof}\smartqed
For all $u\in BV_{rad}(\mathbb{R}^N),$ there exists a sequence $(u_n)\subset BV_{rad}(\mathbb{R}^N)\cap\mathcal{D}(\mathbb{R}^N)$ such that $u_n$  converges strictly in $BV$ and almost everywhere to $u$ (see for instance \cite{Braides1998,ambrosio2000functions}). So, we just have to show the lemma for $u\in BV_{rad}(\mathbb{R}^N)\cap\mathcal{D}(\mathbb{R}^N).$ We denote $u(x)=u(r),$ and we have:
\[
\frac{d}{dr}(r^{N-1}|u|)= \frac{u}{|u|}\frac{du}{dr}r^{N-1}+|u|(N-1)r^{N-2},
\]
and
\[
r^{N-1}|u|=-\int_{r}^{+\infty}\frac{d}{dr}(s^{N-1}|u|)\leq\int_{\mathbb{R}^N}|\nabla u|dx.
 \]
\qed\end{proof}

Sickel, Skrzypczak and Vybiral \cite{sickel2012} studied the properties of radial functions of Besov, Lizorkin-Triebel and $BV$ spaces, generalizing the estimates of this type given by Lions and Strauss. The proof of these inequalities is the first step to get the compactness of the embedding of the following proposition.

\begin{proposition}\label{prop:BVcompactness}
Let $N>1$ and denote $1^*=N/(N-1)$, then the restriction to  $BV_{rad}(\mathbb{R}^N)$ of the embedding $BV(\mathbb{R}^N)\hookrightarrow L^p(\mathbb{R}^N)$ is compact if $p\in (1,1^*).$
\end{proposition}
\begin{proof}\smartqed
Let $(u_n)$ be a bounded sequence in $BV_{rad}(\mathbb{R}^N)$. Up to extraction, there exists a function $u$ belonging to $BV_{rad}(\mathbb{R}^N)$ such that $(u_n)$ tends to $u$ in $L^p_{loc}$ for $p\in (1,1^*).$ Moreover, for $R>0,$ we have:
\[
	\|u_n-u\|_{L^p(\{|x|>R\})}\leq\left(|\nabla u|(\mathbb{R}^N)+|\nabla u_n|(\mathbb{R}^N)\right)^{p-1}\|u-u_n\|_{L^1(\mathbb{R}^N)}R^{-(N-1)(p-1)},
\]
by Lemma \ref{lemma:inequalityforradialfunction}. The result follows immediately.
\qed\end{proof}

\section{The locally compact case of the concentration compactness method in the $BV$ setting}

A general version of the concentration compactness method can be found in the papers of Lions (see for instance \cite{Lions1984-1}) or the book of Struwe \cite{Struwe2008}. Nevertheless, in this paper, we just need the concentration compactness in a simpler setting: we study bounded sequences of  functions in $BV,$ so that, the loss of compactness can just come from the action of the group of translations. Thus, for the reader's convenience, we give in this part, a straightforward adaptation to $BV$ of the presentation of Lewin \cite{Lewin2010} based on the papers of Lions \cite{Lions1984-1} and Lieb \cite{Lieb}. 

The concentraction compactness method has already been used in the $BV$ setting, for instance, by Fusco \cite{Fusco2007} for Sobolev inequalities in $BV$ and by Bucur and Giacomini \cite{Bucur2010}  for the isoperimetric inequality for the Robin eigenvalue problem. But, they both used the arguments of Lions and Struwe.

Let $N>2.$ We begin by defining the highest mass that the limit of translated subsequence can have.
\begin{definition}
Let $(u_n)$ be a bounded sequence in $BV(\mathbb{R}^N),$ we denote:
\[
m(u_n) = \sup\left\{\int_{\mathbb{R}^N}|u|dx: \exists (x_{n_k})\subset\mathbb{R^N}, u_{n_k}(~.~-x_{n_k})\rightarrow u\in BV(\mathbb{R}^N) ~\text{in}~ L^1_{loc}~\right\}.
\]
\end{definition}

The following lemma is related to the vanishing of a sequence.
\begin{lemma}\label{lem:BVvanishing}
Let $(u_n)$ be a bounded sequence in $BV(\mathbb{R}^N),$ we have equivalence between the following points:
\begin{enumerate}
\item \label{lem:BVvanishing1} $m(u_n)=0,$
\item \label{lem:BVvanishing2} for all $R>0$, ~$\underset{n\to+\infty}{\lim}\underset{x\in\mathbb{\mathbb{R}^N}}{\sup}\int_{B(x,R)}|u_n|=0,$
\item \label{lem:BVvanishing3} $u_{n}\rightarrow 0$ strongly in $L^p$ for $p\in(1,1^*),$
\end{enumerate}
where $1^*=N/(N-1).$
\end{lemma}
\begin{proof}\smartqed
Let us assume that \eqref{lem:BVvanishing1} is true. Let $R>0$ and $(x_n)\subset\mathbb{R}^N$ be such that:
\[
\int_{B(x_n,R)}|u_n|dx\geq\sup_{\substack{x\in\mathbb{R}^N}}\int_{B(x,R)}|u_n|dx-1/n.
\]
$(u_n(~.~-x_n))$ is still a bounded sequence in $BV(\mathbb{R}^N)$. Since $m(u_n(~.~-x_n))=0,$ we get that $(u_n(~.~-x_n))$ converges to $0$ in $L^1_{loc}$ and \eqref{lem:BVvanishing2} follows.

Let us assume \eqref{lem:BVvanishing3}. Let $(x_{n_k})\subset\mathbb{R}^N$ be such that 
\[u_{n_k}(~.~-x_{n_k})\rightarrow u\in BV(\mathbb{R}^N) ~\text{in}~ L^1_{loc}~\text{and a.a.}\]
We have:
\[
\|u_{n_k}(~.~-x_{n_k})\|_{L^p(\mathbb{R}^N)}=\|u_{n_k}\|_{L^p(\mathbb{R}^N)}\rightarrow 0,
\]
for all $p\in(1,1^*).$ We immediately get that $u=0$.

Let \eqref{lem:BVvanishing2} be true. We denote $\mathbb{R}^N = \cup_{\substack{z\in\mathbb{Z}^N}}C_z$ with $C_z=\Pi_{\substack{i=1}}^{N}[z_i,z_i+1).$ For $1<p<1^*,$ we have:
\[
\int_{\mathbb{R}^N}|u_n|^pdx = \sum_{\substack{z\in\mathbb{Z}^N}}\int_{C_z}|u_n|^pdx\leq\sum_{\substack{z\in\mathbb{Z}^N}}\|u_n\|^{\theta p}_{L^1(C_z)}\|u_n\|^{(1-\theta)p}_{L^{1^*}(C_z)},
\]
with $1/p= \theta+(1-\theta)/p^*.$ We choose $p$ such that $(1-\theta)p=1,$ that is $p=(N+1)/N\in(1,1^*),$ and we get:
\[
\int_{\mathbb{R}^N}|u_n|^p\leq C\sup_{\substack{z\in\mathbb{Z}^N}}\|u_n\|_{L^1(C_z)}^{\theta p}\|u_n\|_{BV(\mathbb{R}^N)}.
\]
So $(u_n)$ tends to $0$ in $L^p$ and by interpolation inequality in $L^q$ for all $1<q<1^*.$
\qed\end{proof}

When vanishing does not occur, the sequence can converge up to translation and extraction or split into two parts. Lions used the word dichotomy to describe this \cite{Lions1984-1}. This situation is described in the following proposition.
\begin{proposition}\label{prop:BVcc}
Let $(u_n)$ be a bounded sequence in $BV(\mathbb{R}^N)$,  $(R_k)$ and $(R_{k}')$ be two sequences such that  for all $k,$ $0<R_k<R'_k $ and
\[\left\{\begin{array}{ll}
	u_{n}\rightarrow u\in BV(\mathbb{R}^N) ~	&\text{in}~ L^1_{loc}\\
	R_k\rightarrow +\infty.
\end{array}\right.\]
Then, there exists a subsequence $(u_{n_k})$ such that the following properties are true:
\begin{enumerate}
\item $(u_{n_k}\chi_{B(0,R_k)})$ tends to $u$ in $L^p$ if $1\leq p<1^*,$ 
\item $\int_{R_k<|x|<R'_k}|u_{n_k}|dx+|\nabla u_{n_k}|(B(0,R'_k)\backslash B(0,R_k))$ tends to $0$,
\end{enumerate} 
where $\chi_{B(0,R_k)}$ is the characteristic function of the ball $B(0,R_k).$
\end{proposition}

\begin{proof}\smartqed
Let us introduce two Levy's concentration functions:
\[
Q_n(R):=\int_{B(0,R_k)}|u_n|dx ~\text{and}~K_n(R):=|\nabla u_n|(B(0,R_k)).
\] 
These are non-decreasing positive functions such that for all $R>0$ and $n$:
\[
Q_n(R)+K_n(R)\leq\|u_n\|_{BV(\mathbb{R}^N)}\leq C.
\] 
We get for all $R$:
\[
Q_n(R)\rightarrow \int_{B(0,R)}|u|dx=:Q(R),
\]
and up to extraction, there exists $K\in BV(0,+\infty)$ such that:
\[
K_n(R)\rightarrow K(R).
\]
We denote $l:=Ê\lim_{\substack{R\rightarrow+\infty}}K(R).$ There exists a subsequence $(n_k)$ such that:
\[\begin{array}{ll}
|Q_{n_k}(R_k)-Q(R_k)|&+|Q_{n_k}(R'_k)-Q(R'_k)|+\dots\\
&+|K_{n_k}(R_k)-K(R_k)|+|K_{n_k}(R'_k)-K(R'_k)|\leq 1/k.
\end{array}\]
We get:
\[
\left|\int_{B(0,R_k)}|u_{n_k}|dx-\int_{\mathbb{R^N}}|u|dx\right|=|Q_{n_k}(R_k)-Q(\infty)|\leq1/k+\int_{|x|>R_k}|u|dx,
\]
and the theorem of the missing term in the Fatou lemma (see \cite{LiebLoss}) ensures that $(u_{n_k}\chi_{B(0,R_k)})$ tends to $u$ in $L^1.$ This remains true in $L^p$ for $1\leq p< 1^*$ by interpolation. Moreover, we have:
\[\begin{array}{l}
\int_{R_k<|x|<R'_k}|u_{n_k}|dx=Q_{n_k}(R'_k)-Q_{n_k}(R_k)\leq1/k+|Q(R_k)-Q(R'_k)|,\\
|\nabla u_{n_k}|(\{R_k<|x|<R'_k\})=K_{n_k}(R'_k)-K_{n_k}(R_k)\leq1/k+|K(R_k)-K(R'_k)|,
\end{array}\]
so that the second point is also true.  
\qed\end{proof}%

\subsection*{Acknowledgement} The author would like to thank  Eric S\'{e}r\'{e}, Mathieu Lewin, Jimmy Lamboley and the anonymous referee for useful discussions and helpful comments. This work was partially supported by the Grant ANR-10-BLAN 0101 of the French Ministry of research.

\bibliographystyle{plain}      
\bibliography{friedbergLee}   

\end{document}